\documentclass[11pt]{scrartcl}

\usepackage{abstract}

\usepackage[english]{babel}
\usepackage[utf8]{inputenc}		% Direkte Eingabe von Umlauten

\usepackage{url}

\newcommand{\union}{\ensuremath{\cup}}		% Die Vereinigung
\newcommand{\isdef}{\ensuremath{\mathrel{\mathop:}=}}		% Das "ist-definiert"-Zeichen
\usepackage{enumitem}
\let\oldvec\vec % https://tex.stackexchange.com/questions/295293/vec-command-produces-tilde-instead-of-arrow
\usepackage{amsfonts, amsmath, amssymb}
\usepackage{amsthm}	
\theoremstyle{plain}
\usepackage{pgf, tikz}
\usepackage{tikz-cd}
\usepackage{tikz-3dplot}
\usetikzlibrary{calc}
\usetikzlibrary{arrows, arrows.meta}
\usetikzlibrary{positioning}
\usetikzlibrary{fit,decorations.pathreplacing,matrix}
\usetikzlibrary{patterns}
\usepackage{wrapfig}
\usepackage{nicefrac}					% "nette" Brüche erzeugen
\usepackage{color}
\usepackage[colorinlistoftodos,backgroundcolor=yellow, disable]{todonotes}
\renewcommand{\emph}{\textbf}

\newcommand{\ohne}{\ensuremath{\backslash}}

\DeclareMathOperator{\Aut}{Aut}

\DeclareMathOperator{\Hex}{Hex}

\newcommand{\menge}[1]{\ensuremath{\mathbb{#1}}}
\newcommand{\N}{\menge{N}}
\newcommand{\Z}{\menge{Z}}

\newcommand{\e}{\text{e}}				% Das Symbol der eulerschen Zahl

\renewcommand{\phi}{\varphi}
\renewcommand{\epsilon}{\varepsilon}
\renewcommand{\subset}{\subseteq}

\newtheorem{lem}{Lemma}[section]		% Lemma
\newtheorem{ex}[lem]{Example}		% Beispiel
\newtheorem{rem}[lem]{Remark}
\newtheorem{theo}[lem]{Theorem}
\newtheorem*{theos}{Theorem}			% Satz
\newtheorem{cor}[lem]{Corollary}		% Korollar
\newtheorem{prop}[lem]{Proposition}		% Proposition
\newtheorem{defi}[lem]{Definition}	% Definition
{\begin{proof}[Well--defined]}
	{\end{proof}}

\newcommand{\K}[1]{\mathbb{K}(#1)}

\newcommand{\isom}{%
			\mathrel{\ooalign{\hss\hbox{\resizebox{.03\hsize}{.006\vsize}{$\rightarrow$}}\hss\cr%
			\kern0.1ex\raise0.45ex\hbox{\resizebox{.02\hsize}{.004\vsize}{$\sim$}}}}}
		
\newcommand{\iso}[1]{\overset{#1}{\isom}}

\newcommand{\neig}[2]{N_{#1}\left[#2\right]}
\newcommand{\cneig}[2]{N_{#1}^{\cap}\left[#2\right]}

\newcommand{\mammal}{pika}

\newcommand{\mammals}{pikas}
\newcommand{\Mammals}{Pikas}
\newcommand{\ccon}{$k$-convergent}
\newcommand{\cdiv}{$k$-divergent}
\newcommand{\thcir}{three-circle}
\newcommand{\thcirs}{three-circles}

\newcommand{\mathemph}[1]{\mathbf{#1}}

\newcommand{\HexagonalCoordinates}[2]{
	\foreach \i in {0,...,#1}{
		\foreach \j in {0,...,#2}{
			\coordinate (A\i\j) at ($(2*\i,0)+(60:2*\j)$);
			\coordinate (D\i\j) at ($(A\i\j)+(4/3,0)+(60:4/3)$);
			\coordinate (U\i\j) at ($(A\i\j)+(2/3,0)+(60:2/3)$);
		}
	}
}
\author{Markus Baumeister, Anna M. Limbach}
\title{Clique dynamics of locally cyclic graphs with $\delta\geq 6$}

\begin{document}
\maketitle

%\input{CliqueGraphPart1.tex}
% flatex input: [CliqueGraphConstruction.tex]
\listoftodos

\begin{abstract}
We prove that the clique graph operator $k$ is divergent on a locally cyclic 
graph $G$ (i.\,e. $N_G(v)$ is a circle) with minimum degree 
$\delta(G)=6$ if and only if $G$ is $6$-regular.
The clique graph $kG$ of a graph $G$ has the maximal complete subgraphs of 
$G$ as vertices, and the edges are given by non-empty intersections.
If all iterated clique graphs of $G$ are pairwise non-isomorphic, the graph $G$
is \cdiv; otherwise, it is \ccon.

To prove our claim, we explicitly construct the iterated clique graphs of those infinite locally cyclic graphs with $\delta\geq6$ which induce 
simply connected simplicial surfaces. These graphs are \ccon\enspace if the 
size of triangular-shaped subgraphs of a specific type is bounded from above. 
We apply this criterion by using the universal cover of 
the triangular complex of an arbitrary finite 
locally cyclic graph with $\delta= 6$, which shows our divergence 
characterisation.

\end{abstract}

\section{Introduction}\label{Sect_Introduction}
Applied to a graph $G$, the clique graph operator constructs its clique graph $kG$.
The vertices of $kG $ are the maximal complete subgraphs of $G$, called 
\emph{cliques}. These cliques are adjacent in $kG$ if they intersect %non-trivially 
in $G$. In 1972, Hedetniemi and Slater first studied line graphs and triangle free graphs 
using the clique graph operator \cite{hedetniemi1972line}.
We are interested in locally cyclic graphs,
which means that the 
set of vertices incident to a given vertex $v$ always induces a circle. Popular locally cyclic graphs are the octahedron,
the icosahedron, and the hexagonal grid, which are displayed in Figure \ref{fig_loccyclic}.
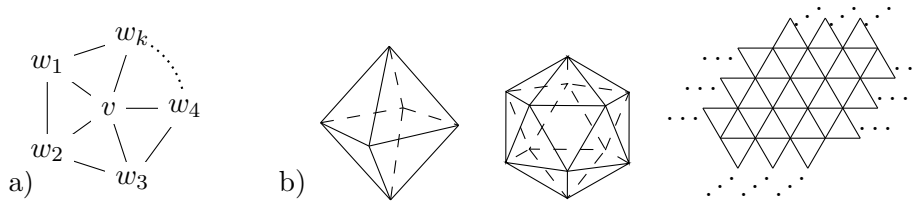
\begin{figure}[htbp]
    \begin{center}
        \begin{tikzpicture}[scale=0.5]
            \def\rad{2}
            \coordinate (Z) at (0,0);
            \foreach \i in {0,...,4}{
                \coordinate (R\i) at (\i*72:\rad);
            }		
            \foreach \i in {1,...,4}{
                \draw (Z) -- (R\i);
                \pgfmathparse{mod(int(1+\i),5)}
                \pgfmathsetmacro{\nxt}{\pgfmathresult}
                \draw (R\i) -- (R\nxt);
                %		\draw (R\i) -- (A\i) -- (R\nxt);	
            }
            \draw (R0) -- (Z);
            \node at (-2.3,-2) {a)};
            \node[fill=white] at (Z) {$v$};
            
            \node[fill=white] at (R1) {$w_k$};	
            \node[fill=white] at (R2) {$w_1$};
            \node[fill=white] at (R3) {$w_2$};
            \node[fill=white] at (R4) {$w_3$};
            \node[fill=white] at (R0) {$w_4$};				
            
            \def\deg{15}
            \draw[dotted,thick] (\deg:2) arc (\deg:72-\deg:\rad);
            \end{tikzpicture}
            \hspace{0.5cm}
            \begin{tikzpicture}[scale=0.43]%
            \node at (3,0.5) {b)};
            \draw [color=black] (5.96,4.74)-- (3.86,2.38); 
            \draw [color=black] (3.86,2.38)-- (5.34,1.66); 
            \draw [color=black] (5.34,1.66)-- (5.96,4.74); 
            \draw (5.34,1.66)-- (6.,0.); 
            \draw (5.34,1.66)-- (8.08,2.3); 
            \draw (3.86,2.38)-- (6.,0.); 
            \draw (6.,0.)-- (8.08,2.3); 
            \draw (5.96,4.74)-- (8.08,2.3); 
            \draw [dash pattern=on 5pt off 5pt] (5.96,4.74)-- (6.38,2.84); 
            \draw [dash pattern=on 5pt off 5pt] (6.38,2.84)-- (8.08,2.3); 
            \draw [dash pattern=on 5pt off 5pt] (3.86,2.38)-- (6.38,2.84); 
            \draw [dash pattern=on 5pt off 5pt] (6.38,2.84)-- (6.,0.);
        \end{tikzpicture}
        \hspace{0.1cm} 
        \def \phi {1.617}
        \begin{tikzpicture}[
            x={(-0.86in, -0.5in)}, y = {(0.86in, -0.5in)}, z = {(0, 1in)},
            rotate = 22,
            scale = 0.1,
            %every node/.style = {
            %	circle, fill = blue!20, inner sep = 0pt, minimum size = 0.5cm
            %},
            foreground/.style = { },
            background/.style = { dash pattern=on 5pt off 5pt}
            ]
            \coordinate (9) at (0, -\phi*\phi,  \phi);
            \coordinate (8) at (0,  \phi*\phi,  \phi);
            \coordinate (12) at (0,  \phi*\phi, -\phi);
            \coordinate (5) at (0, -\phi*\phi, -\phi);
            \coordinate (7) at ( \phi, 0,  \phi*\phi);
            \coordinate (3) at (-\phi, 0,  \phi*\phi);
            \coordinate (6) at (-\phi, 0, -\phi*\phi);
            \coordinate (4) at ( \phi, 0, -\phi*\phi);
            \coordinate (2) at ( \phi*\phi,  \phi, 0);
            \coordinate (10) at (-\phi*\phi,  \phi, 0);
            \coordinate (1) at (-\phi*\phi, -\phi, 0);
            \coordinate (11) at ( \phi*\phi, -\phi, 0);
            
            \draw[foreground] (10) -- (3) -- (8) -- (10) -- (12) -- (8);
            \draw[foreground] (4) -- (12) -- (2) -- (4) -- (11) -- (2);
            \draw[foreground] (9) -- (3) -- (7) -- (9) -- (11) -- (7);
            \draw[foreground] (7) -- (8) -- (2) -- cycle;
            \draw[background] (12) -- (6) -- (10) -- (1) -- (6) -- (5) -- (1)
            -- (9) -- (5) -- (11);
            \draw[background] (5) -- (4) -- (6);
            \draw[background] (3) -- (1);
            \foreach \n in {1,...,12}
            \node at (\n) {};
        \end{tikzpicture}\hspace{0.1cm}
        \begin{tikzpicture}[scale=0.23]
            \HexagonalCoordinates{7}{7}
            \draw (A33) -- (A37);
            \draw (A42) -- (A47);
            \draw (A52) -- (A57);
            \draw (A62) -- (A66);
            \draw (A33) -- (A73);
            \draw (A24) -- (A74);
            \draw (A25) -- (A75);
            \draw (A26) -- (A66);
            \draw (A42) -- (A24);
            %\draw (A20) -- (A02);
            \draw (A52) -- (A25);
            \draw (A62) -- (A26);
            \draw (A63) -- (A36);
            \draw (A73) -- (A37);
            \draw (A74) -- (A47);
            \draw (A75) -- (A57);
            \node at (U75) {$\dots$};
            \node at (U74) {$\dots$};
            \node at (U73) {$\dots$};
            \node at (U57) {\reflectbox{$\ddots$}};
            \node at (U47) {\reflectbox{$\ddots$}};
            \node at (U37) {\reflectbox{$\ddots$}};
            \node at (D15) {$\dots$};
            \node at (D14) {$\dots$};
            \node at (D13) {$\dots$};
            \node at (D51) {\reflectbox{$\ddots$}};
            \node at (D41) {\reflectbox{$\ddots$}};
            \node at (D31) {\reflectbox{$\ddots$}};
        \end{tikzpicture}
    \end{center}
    \caption{a) The cyclic neighbourhood of $v$, b) Octahedron, icosahedron, hexagonal grid}\label{fig_loccyclic}
\end{figure}
 For minimum degree $\delta$ of at least 4, they can be described as 
 Whitney triangulations of surfaces, which were investigated for example 
in \cite{larrion2003clique},\cite{larrion2002whitney}, and 
\cite{larrion2006graph}.
In 1999, Larrión and Neuman-Lara showed that some \mbox{$6$-regular}
triangulations of the torus are \cdiv\enspace \cite{larrion1999clique} 
and, in 2000, they generalised this result to every \mbox{$6$-regular}
locally cyclic graph  graph \cite{larrion2000locally}.
Furthermore,  Larriòn, Neumann-Lara, and Piza\~na \cite{larrion2002whitney} 
showed that graphs in which every open neighbourhood of a vertex has a girth of at least $7$ are \ccon. 
Thus, locally cyclic graphs of minimum degree $\delta$ of at least $7$ are \ccon. The question remains whether every non-regular locally cyclic 
graph with  $\delta=6$ is \ccon. In this paper, 
we provide the affirmative answer by generalising the approach 
from \cite{larrion2000locally} and using the theoretical 
background from \cite{rotman1973covering}.

In the remainder of this paper, a graph is not necessarily considered finite. 
%For a graph $G=(V_G,E_G)$, 
% the \emph{closed neighbourhood} of $M\subset V_G$ in $G$ is given by the induced subgraph
%\begin{equation*}
%    \mathemph{N_G[M]}\isdef
%    G[y\in V_G\mid y\in M\text{ or }\exists x\in M\colon xy\in E_G],
%\end{equation*}
%and the \emph{common neighbourhood} of $M$ is 
%\begin{equation*}
%    \mathemph{N_G^{\cap}[M]}\isdef G[y\in V_G\mid y\in M\text{ or } \forall x\in M\colon xy\in E_G ].
%\end{equation*}
%For a subgraph $H$ of $G$, $\mathemph{N_G[H]}\isdef N_G[V_H]$ and 
%$\mathemph{N_G^{\cap}[H]}\isdef N_G^{\cap}[V_H]$. Furthermore, for a vertex $v\in V_G$,\todo{why don't use the neighbourhood-makros?}
%the \emph{closed neighbourhood} is given by $\mathemph{N_G[v]}\isdef N_G[\{v\}]$ and the open neighbourhood 
%is given by $N_G(v)=G[y\in V_G\mid vy\in E_G].$\todo{maybe the open neighburhood should be defined when defining locally cyclic}
%
%%If $H$ is a subgraph of $G$, the neighbourhood of $H$ in $G$ is given by $N_G[H]=G[y\in V_G\mid \exists x\in V_H\colon xy\in E_G]$.
%%The common neighbourhood of $H$ in $G$ is given by $ N_G^{\cap}[H]=G[y\in V_G\mid \forall x\in V_H\colon xy\in E_G]$.
%For the two graphs $G = (V_G, E_G)$ and $H = (V_H, E_H)$,
%the graph $\mathemph{G\setminus H}$ is defined as 
%$(V_G\ohne V_H, \{ xy \in E_G \ohne E_H \mid x \not\in V_H \wedge y \not\in V_H \})$. \todo{we need this in section four. Shall we move it there?}
%
% 
%%If $H$ is a subgraph of $G$, the neighbourhood of $H$ in $G$ is given by 
%%$N_G(H)=G[xy\in E_G\mid x\in V_H]$. 
%
Furthermore, we extend the incidence terminology of a graph to \thcirs. Thus, the three vertices and the three edges of a 
\thcir\enspace are each \emph{incident} to the \thcir\ itself.
Throughout the paper, we use different kinds of neighbourhoods, which are defined in the appendix.

\section{Overview and basic concepts}
The focus of this paper consists in proving our main result:
\begin{theos}[Main result]
    Let $G$ be a finite, locally cyclic graph with minimum degree $\delta=6$.
    The clique graph operator diverges on $G$ if and only if $G$ is $6$-regular.
\end{theos}

%\noindent This main result is proven combining two theorems:
%
%\begin{theos}
%	 Let $G$ be a triangularly simply connected locally cyclic graph with 
%	minimum degree at least $6$. 
%	If there is an $m\geq 0$ such that 
%	$\Delta_m$ cannot be embedded into $G$, the clique operator is 
%	convergent on $G$.
%\end{theos}
%
%\begin{theos}
%	Let $G$ be a locally cyclic graph with minimum degree $\delta=6$. If the universal cover of $G$ is \ccon, the graph $G$ itself is \ccon.
%\end{theos}

\noindent This paper is based on two core insights:
\begin{enumerate}
    \item The explicit description of iterated clique graphs
        is massively simplified if we restrict ourselves to 
        triangularly simply connected graphs (there, triangular
        substructures do not ``self-overlap'').
        Subsection \ref{Subsect_SimplyConnected} gives a short
        refresher on the definitions. In Section \ref{Sect_UniversalCover},
        we extend the result to general locally cyclic graphs with 
        $\delta =6$ using universal covers.
    \item Clique graphs grow only in regions with ``many'' vertices
        of degree 6. More precisely, these vertices have to form
        a triangular-shaped structure. We can also parametrise these regions by
        barycentric coordinates. The necessary
        formalism is given in Subsection \ref{Subsect_Hexagonal}.
        This way, we combinatorially encode the adjacencies
        in the iterated clique graphs in Section \ref{Sect_Graph}. 
\end{enumerate}

\noindent For simplification, we shrink the specifications of the centrally discussed object with a new definition.
\begin{defi}
    A \emph{\mammal} is a triangularly simply connected locally
    cyclic graph $G$ with minimum degree $\delta=6$.
\end{defi}

\noindent For \mammals, we head for the following theorem. 

\begin{theos}[Main Theorem for \Mammals]
    Let $G$ be a \mammal.
    If there is an $m\geq 0$ such that 
    the triangular-shaped graph of side length $m$ (see Definition 
    \ref{Def_HexagonalGrid}) cannot be embedded into $G$, the clique operator is 
    convergent on $G$.
\end{theos}

\noindent We prove the main theorem for \mammals\enspace by induction. 
For a \mammal\enspace $G$ and for every $n\in\Z_{\geq 0}$, we define
a graph $G_n$, beginning with $G_0= G$. 
We construct all the cliques of $G_n$ and their intersection, which yields 
$G_{n+1}\cong k(G_n)$ and, by induction, $k^nG=G_n$. Thus, $G$ is clique 
convergent if and only if the sequence $G_n$ converges.

Before diving into this line of arguments, in Section \ref{Sect_Topology}, 
we discuss some intricate topological arguments that are based
on the discrete curvature of a locally cyclic graph. 
The results will be used to show that  all cliques in $G_n$ are of 
one of the types we describe and that the adjacencies in $G_{n+1}$ 
correspond to the intersections of cliques of $G_n$.  

In Section \ref{Sect_Graph}, we define the graph $G_n$ and construct 
two types of cliques in $G_n$. In Section \ref{Sect_ChartExtension}, 
our goal is to ensure the existence of a combined parametrisation of 
intersecting triangular-shaped subgraphs in the general case and to 
handle the exceptional cases. The results of this discussion are used 
in Section \ref{Sect_CliqueCompleteness} to prove that $G_n$ has no 
more cliques than the ones from Section \ref{Sect_Graph}. In Section 
\ref{Sect_CliqueIntersections}, we finish the inductive proof for 
$G_n\cong k^nG$ by describing the clique intersections in $G_n$ 
through the vertex adjacencies in $G_{n+1}$ and prove the main theorem 
for \mammals. In Section \ref{Sect_UniversalCover}, we deduce a 
convergence criterion that does not rely on the simple connectivity
of $G$. In the special case of finite $G$ we conclude the main result. 
Section \ref{Sect_FurtherResearch} includes our conjecture about 
the infinite case and our suggestions for further research questions.

%Version with topology in the triangular complex
\subsection{Review: Simple Connectivity}\label{Subsect_SimplyConnected}
In this subsection, we review topological aspects of locally cyclic
graphs.
The definition of a path we use here originates from topological settings. In graph theoretic literature, those paths would be called walks. A \emph{path}	$p = x_0x_1\ldots x_k$ in a graph $G$ is a finite sequence of vertices such that $x_ix_{i+1} \in E$ for all $0 \leq i < k$.
The \emph{length} of a path is the number of contained edges.
Let $G$ be a locally cyclic graph. Following 
\cite{larrion2000locally}, its \emph{triangular complex} is the
simplicial complex $\mathemph{\K{G}}$, whose simplices are the vertices, edges,
and \thcirs\enspace of $G$. In this way, the \thcirs\enspace of $G$ become the facets 
of $\K{G}$ and, from now on, we will call them the \emph{facets} of $G$, too.
%Since $G$ is locally cyclic, the geometric realisation $\mathemph{\lvert\K{G}\rvert}$ 
%is a closed topological surface. Thus, we can
%interpret $\K{G}$ as a triangulation of the surface $|\K{G}|$.

Like in \cite{rotman1973covering}, we call two paths $\alpha$ and $\beta$ in $G$ \emph{equivalent} 
if we can reach one path from the other by applying a finite number of elementary moves 
consisting in replacing two consecutive path edges $uv$ and $vw$  by the 
edge $uw$ or the other way around whenever $\{u,v,w\}$ is a facet. 
The complex $\K{G}$ is called 
\emph{simply connected} if $G$ is connected --- i.\,e. for every pair of 
vertices, there is a path in $G$ connecting them ---,
 and if every closed path 
$\alpha$ is equivalent to a path consisting of a single vertex (which is 
the origin and end of $\alpha$). We call a locally cyclic graph 
$G$ \emph{triangularly simply connected} if $\K{G}$ is simply connected.

%old version
%In this subsection, we review topological aspects of locally cyclic
%graphs.
%Let $G$ be a locally cyclic graph. Following 
%\cite{larrion2000locally}, its \emph{triangular complex} is the
%simplicial complex $\mathemph{\K{G}}$, whose simplices are the vertices, edges,
%and \thcirs\enspace of $G$. In this way, the \thcirs\enspace of $G$ become the facets 
%of $\K{G}$ and, from now on, we will call them the \emph{facets} of $G$, too.\todo{ should we do this more often?}
%Since $G$ is locally cyclic, the geometric realisation $\mathemph{\lvert\K{G}\rvert}$ 
%is a closed topological surface. Thus, we can
%interpret $\K{G}$ as a triangulation of the surface $|\K{G}|$.
%
%Like in \cite{rotman1973covering}, we call two paths $\alpha$ and $\beta$ in $G$ \emph{equivalent} 
%if we can get from one path to the other by applying a finite number of elementary moves 
%consisting in replacing two consecutive path edges $uv$ and $vw$  by the 
%edge $uw$ or the other way around whenever $\{u,v,w\}$ is a facet. 
%The complex $\K{G}$ is called 
%\emph{simply connected} if $G$ is connected -- i.\,e. for every pair of 
%vertices, there is a path in $G$ connecting them --,
%and if every closed path 
%$\alpha$ is equivalent to a path consisting of a single vertex (which is 
%the origin and end of $\alpha$). We call a locally cyclic graph 
%$G$ \emph{triangularly simply connected} if $|\K{G}|$ is simply connected.

\subsection{Hexagonal Grid}\label{Subsect_Hexagonal}
If a locally cyclic graph has an area of vertex degree $6$, the graph locally looks like the hexagonal grid, a term we will define now.
\begin{defi}\label{Def_HexagonalGrid}
    We define the coordinate set 
    \begin{equation*}
        \mathemph{\oldvec{D}_0} \isdef \{ (1,-1,0), (1,0,-1), (-1,1,0), (0,1,-1), (-1,0,1), (0,-1,1) \}.
    \end{equation*}
    For $m \in \Z$, the \emph{hexagonal grid of height} $\mathemph{ m}$ is the graph $\mathemph{Hex_m} = (V_m,E_m)$ with
    \begin{align*}
        \mathemph{V_m} &\isdef \{ (x_1,x_2,x_3) \in \Z^3 \mid x_1+x_2+x_3 = m \} \text{ and}\\ 
        \mathemph{E_m} &\isdef \{ \{x,y\} \subset V_m \mid x-y \in \vec{D}_0 \}.
    \end{align*}
    For $m\geq 0$, we denote the \emph{triangular-shaped graph of side length} $\mathemph{m}$, which is defined as $\Hex_m[V_m\cap\Z_{\geq 0}^3]$, by $\mathemph{\Delta_{m}}$.
    Figure \ref{Fig_Delta_m} shows the smallest five of those subgraphs.

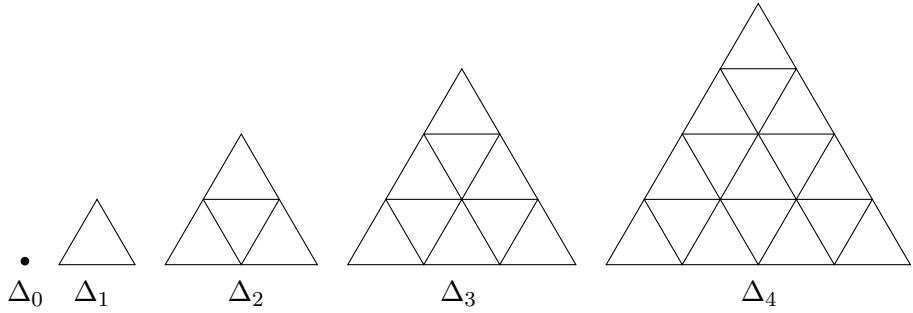
\begin{figure}[htbp]
    \begin{center}
        \quad\quad
        \begin{tikzpicture}[scale=0.5]
            \HexagonalCoordinates{1}{1}
            \node[draw,fill=black,inner sep=1pt, shape=circle](x){};
            \end{tikzpicture}\quad
            \begin{tikzpicture}[scale=0.5]
            \HexagonalCoordinates{1}{1}
            \draw (A00) -- (A10) -- (A01) --cycle;
        \end{tikzpicture}\quad
        \begin{tikzpicture}[scale=0.5]
            \HexagonalCoordinates{2}{2}
            \draw (A00) -- (A20) -- (A02) --cycle;
            \draw (A10) -- (A01) -- (A11) --cycle;
        \end{tikzpicture}\quad
        \begin{tikzpicture}[scale=0.5]
            \HexagonalCoordinates{3}{3}
            \draw (A00) -- (A30) -- (A03) --cycle;
            \draw (A01) -- (A21) -- (A20) --(A02) -- (A12) -- (A10) --cycle;
        \end{tikzpicture}\quad
        \begin{tikzpicture}[scale=0.5]
            \HexagonalCoordinates{4}{4}
            \draw (A00) -- (A40) -- (A04) --cycle;
            \draw (A01) -- (A31) -- (A30) --(A03) -- (A13) -- (A10) --cycle;
            \draw (A02) -- (A20) -- (A22) --cycle;
        \end{tikzpicture}
        \nopagebreak
        $\Delta_0\quad\Delta_1\quad\quad\quad\quad \Delta_2\quad\quad\quad\quad\quad\quad \Delta_3\quad\quad\quad\quad\quad\quad\quad\quad\quad \Delta_4\quad\quad\quad$
    \end{center}
    \caption{$\Delta_m$ for $m\in \{0,\ldots,4\}$}
    \label{Fig_Delta_m}
\end{figure}
%\noindent A \emph{graph homomorphism} $\phi\colon G\to H$ is any 
%adjacency-preserving vertex map, i.\,e. 
%$uv\in E(G)\ \Rightarrow\ f(u)f(v)\in E(H)$, see \cite[Section 2, p. 161]{larrion2000locally}. Injective homomorphisms are called monomorphisms, bijective homomorphisms are called isomorphisms, if their inverse is also a homomorphism.   
\end{defi}

For a locally cyclic graph $G$, a \emph{hexagonal chart} is a graph
isomorphism $\mu: H \to F$ (also written $H \iso{\mu} F$) with
vertex-induced subgraphs $H \subseteq \Hex_m$ and $F\subseteq G$.
If $F \cong \Delta_m$, we call it \emph{standard chart}.

Since the symmetric group on three points acts on the hexagonal grid
by coordinate permutations, every subgraph $F \cong \Delta_m$ with
$m \geq 1$ has six standard charts.

For $(t_1,t_2,t_3)\in\Z^3$, we define the \emph{triangle inclusion map}:
    \begin{equation*}
        \mathemph{\Delta_m^{t_1,t_2,t_3}} : \Delta_m \to \Hex_{m+t_1+t_2+t_3}, \qquad
        (a_1,a_2,a_3) \mapsto (a_1+t_1,a_2+t_2,a_3+t_3).
    \end{equation*}

\section{Topology}\label{Sect_Topology}
We translate $\Delta_m$-shaped graphs into the setting of locally cyclic graphs.
    A \emph{locally cyclic graph with boundary} is a simple
   graph $G=(V,E)$ such 
    that for every vertex $v\in V$ the (open)
    neighbourhood $N_G(v)$ is either a circle graph or a path graph. %=\{w\in V\mid vw\in E\}$
    If $N_G(v)$ is a circle, 
    $v$ is called an \emph{inner vertex} of $G$; otherwise, $v$ is called 
    a \emph{boundary vertex.}
    
    An edge $xy\in E$ is called an \emph{inner edge} if its incident 
    vertices $x$ and $y$ have 
    two common neighbours, and a \emph{boundary edge}, if not. The \emph{boundary graph} $\mathemph{\partial G}$
    is the subgraph of $G$ consisting of the boundary vertices and the boundary edges.
    $G$ is called \emph{locally cyclic} if $\partial G=\emptyset$.

    The boundary graph $\partial G$ is well-defined, since for every inner 
    vertex $x$ and every edge $xy$, the vertex $y$ lies in the cyclic 
    neighbourhood $N_G(x)$ and has, therefore, two neighbours in 
    $N_G(x)$. Thus, $xy$ is an inner edge. Conversely, there
    exist inner edges that are incident to only boundary vertices.

\subsection{Straight Paths}\label{Subsect_StraightPaths}
A monomorphism of locally cyclic graphs with boundary preserves vertex degrees 
of inner vertices. Furthermore, the number of incident facets on either side of a 
path is preserved, too. We formalise this by the concept of \emph{path degree}:

    Let $p = x_0x_1\ldots x_k$ be a path in a locally cyclic graph $G$ and 
    consider a vertex $x_i$
    for $0 < i < k$.
    \begin{itemize}
        \item If $x_i$ is an inner vertex, $N_G(x_i)$ is a circle, say
            of length $L$, and 
            marking $x_{i-1}$ and $x_{i+1}$ splits the circle into
            two paths of lengths $l_1$ and $l_2$, satisfying $l_1+l_2=L$.
            The \emph{path degree} $\mathemph{deg_G^p(x_i)}$ is defined as
            $\{l_1,l_2\}$, as is visualized in Figure \ref{fig_pathdeg}.
        \item If $x_i$ is a boundary vertex, $N_G(v)$ is a path graph
            containing a unique shortest path $q$ from $x_{i-1}$ to $x_{i+1}$
            with length $l$.
            The \emph{path degree} $\mathemph{deg_G^p(x_i)}$ is defined as $\{l\}$.
    \end{itemize}
    The concept of path degrees is illustrated in Figure \ref{fig_pathdeg}.
    The path $p$ is called \emph{straight}
    if $3$ is contained in $\deg_G^p(x_i)$ for every $0 < i < k$. 
     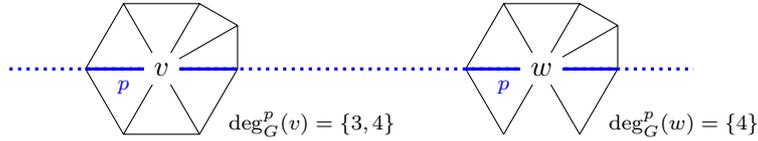
\begin{figure}[htbp]
            \centering
            \begin{tikzpicture}[scale=0.5]
            \HexagonalCoordinates{10}{4}
            \draw
            (A11) -- (A20) -- (A30) -- (A31) node[draw=none,fill=none,font=\scriptsize,midway,below right] {$\deg_G^p(v)=\{3,4\}$} -- (D21) -- (A22) -- (A12) -- (A11)
            (A20) -- (A22)
            (A30) -- (A12)
            (A21) -- (D21);
            \draw
            [very thick, blue]
            (A11) -- (A21) node[draw=none,fill=none,font=\scriptsize,midway,below] {$p$} -- (A31);
            %    	\draw
            %    	[very thick, blue, dotted]
            %    	(A01) -- (A51);
            
            %	\node[fill=white] at (A21) {$v$};
            %    	\end{tikzpicture}
            %    	\begin{tikzpicture}[scale=0.5]
            %    	\HexagonalCoordinates{4}{4}
            \draw
            (A61) -- (A70);
            \draw (A80) -- (A81) node[draw=none,fill=none,font=\scriptsize,midway,below right] {$\deg_G^p(w)=\{4\}$} -- (D71) -- (A72) -- (A62) -- (A61)
            (A70) -- (A72)
            (A80) -- (A62)
            (A71) -- (D71);
            \draw
            [very thick, blue]
            (A61) -- (A71) node[draw=none,fill=none,font=\scriptsize,midway,below] {$p$} -- (A81);
            \draw
            [very thick, blue, dotted]
            (A01) -- (A91);
            \node[fill=white] at (A21) {$v$};
            \node[fill=white] at (A71) {$w$};
            \end{tikzpicture}
            %    	\begin{tikzpicture}[scale=0.5]
            %    	\HexagonalCoordinates{4}{4}
            %    	\draw
            %    	(A11) -- (A20);
            %    	\draw (A30) -- (A31) node[draw=none,fill=none,font=\scriptsize,midway,below right] {$\deg_G^p(v)=\{4\}$} -- (D21) -- (A22) -- (A12) -- (A11)
            %    	(A20) -- (A22)
            %    	(A30) -- (A12)
            %    	(A21) -- (D21);
            %    	\draw
            %    	[very thick, blue]
            %    	(A11) -- (A21) node[draw=none,fill=none,font=\scriptsize,midway,below] {$p$} -- (A31);
            %    	\draw
            %    	[very thick, blue, dotted]
            %    	(A01) -- (A41);
            %    	
            %    	\node[fill=white] at (A21) {$v$};
            %    	\end{tikzpicture}
            \caption{The path degrees of the inner vertex $v$ and the boundary vertex $w$. Since  the path degree $\deg_G^p(w)$ does not contain $3$, $p$ is not straight.}\label{fig_pathdeg}
    \end{figure}

\noindent As an important application, we construct the straight paths within $\Delta_m$.
\begin{rem}\label{Rem_StraightPathsInTriangle}
    Up to symmetry (see Subsection \ref{Subsect_Hexagonal}),
    the maximal straight paths with length at least $m-2$ in 
    $\Delta_m$ (with $m \geq 3$) are the following, 
    depicted in Figure \ref{Fig_alphabetagamma}:
    \begin{enumerate}
	\item For length $m$, we have $\alpha: \{0,\dots,m\}\to\Z^3$ with 
	    $t \mapsto (m-t,t,0)$.
	\item For length $m-1$, we have $\beta: \{0,\dots,m-1\}\to\Z^3$ 
	    with $t \mapsto (m-1-t,t,1)$.
	\item For length $m-2$, we have $\gamma: \{0,\dots,m-2\}\to\Z^3$ 
	    with $t \mapsto (m-2-t,t,2)$.
    \end{enumerate}
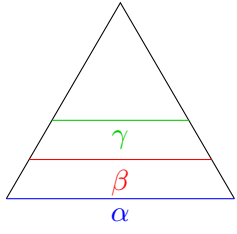
\begin{figure}[htbp]
	\centering
	\begin{tikzpicture}[scale=0.3]
		\HexagonalCoordinates{5}{5}
		%\node[label=right:\textcolor{blue}{$\alpha$}]  at (A40) {};
		\draw (A00) -- (A05) -- (A50);
		\draw[blue] (A00) --node[below] {$\alpha$} (A50);
		\draw[red] (A01) --node[below] {$\beta$} (A41);
		\draw[green!80!black] (A02)--node[below] {$\gamma$} (A32);
	\end{tikzpicture}
	\caption{The maximal straight paths in $\Delta_m$}\label{Fig_alphabetagamma}
\end{figure}
\end{rem}
\begin{proof}
    The boundary $\partial \Delta_m$ consists of six straight
    paths of length $m$, given by $\alpha$ and its images under coordinate permutations.
    Since any other straight path contains an inner vertex and
    inner vertices have degree 6, the value of one of
    the three coordinates is constant along the path (compare Subsection
    \ref{Subsect_StraightPaths}).
    Without loss of generality (see Subsection \ref{Subsect_Hexagonal}),
   let this be the third coordinate. This way, we receive $\beta$ and $\gamma$.
\end{proof}

\subsection{Topological consequences of $\delta=6$}

To understand the structure of a \mammal\enspace $G$ properly, we need to make sure that the boundary vertices of a $\Delta_m$-shaped subgraph are not connected by paths in an unexpected way.
\begin{lem}\label{lem_topology}
    For every induced subgraph $S\cong\Delta_m$ of 
    %simply connected locally cyclic graph $G$ with minimum degree $\delta=6$, \todo{is is a pika, pull definition to here?}
        $G$ the following statements hold:
    \begin{enumerate}
	\item\label{top_a} Any edge incident to two boundary vertices of $S$ lies in $S$.
        \item\label{top_b} Let $v_0v_1v_2$ be a path with $v_0, v_2 \in \partial S$ and $v_1 \notin S$. 
            Then, either $v_0 = v_2$ or 
            $\{v_0,v_1,v_2\}$ is a facet (i.e. $v_0v_2$ is
            a boundary edge). 
	\item\label{top_c} Let $v_0v_1v_2v_3$ be a non-repeating 
            path with $v_0, v_3 \in \partial S$ and $v_1,v_2\notin S$
            such that neither $\{v_0,v_1,v_3\}$ nor $\{v_0,v_2,v_3\}$ are facets.
            Then, there exists a boundary vertex $v \in \partial S$ 
            such that $\{v,v_1,v_2\}$ is a facet.
    \end{enumerate} 
\end{lem}
\begin{proof}
	See Appendix \ref{Sec_Proofsfromtopology}.
\end{proof}
\noindent This helps proving two more auxiliary lemmas.
\begin{lem}\label{Lem_NeighbourhoodIntersection}
    Let $m \geq 1$ and $\Delta_m \cong S = S_1 \union S_2 \union S_3$ with $S_i \cong \Delta_{m-1}$.
    Then, $\neig{G}{S_1} \cap \neig{G}{S_2} \cap \neig{G}{S_3} \subset S$.
\end{lem}
\begin{proof}
	See Appendix \ref{Sec_Proofsfromtopology}.
\end{proof}

\begin{lem}\label{Lem_NeighbourhoodCycle}
    Let $S \cong \Delta_m$. Then, $\neig{G}{S} \ohne S$ is a cycle and 
    vertices in $\neig{G}{S}\ohne S$ are incident to at most three faces
    in $\neig{G}{S}$.
\end{lem}
\begin{proof}
	See Appendix \ref{Sec_Proofsfromtopology}.
\end{proof}

\section{The graph $G_n$}\label{Sect_Graph}

\noindent We construct a graph sequence $G_n$ for every \mammal\enspace $G$ in a geometric way (see Def. \ref{Def_theCliqueGraph}).
%iterated clique graphs $k^nG$ explicitly.
%This is possible since $k^nG$ can be described geometrically. To keep
%the geometric and recursive definitions separate, we define the geometric
%clique graph $G_n$ in Definition \ref{Def_theCliqueGraph}. Then, we need to
%show that $G_0 = G$ and $kG_n = G_{n+1}$ to prove the equivalence of the
%two definitions.

In Subsection \ref{Subsect_Triangles}, we prove some general
statements concerning the relation between different triangular-shaped subgraphs.
These will be helpful in all further analyses.
In Subsection \ref{Subsect_CliqueConstruction}, 
we construct some cliques of $G_n$ explicitly. This is the first step on our way to  prove $G_n\cong k^nG$ inductively. 
%In Section
%\ref{Sect_CliqueCompleteness},
%we show that $G_n$ has no more than those cliques. Finally, we show that
%the edges of $kG_n$ are the edges of $G_{n+1}$ in Section \ref{Sect_CliqueIntersections}.\todo{Look at this paragraph again. Are there redundancies with Section 2? }

If not otherwise stated, from now on $G$ will always refer to a
\mammal\enspace and $G_n$ will be the geometric clique graph
defined in Definition \ref{Def_theCliqueGraph}.

\begin{defi}\label{Def_theCliqueGraph}
    Let $G$ be a \mammal. For a non-negative integer $n$, the 
    \emph{geometric clique graph} $\mathemph{G_n}$ has the following form:
    \begin{itemize}
        \item Its vertices are the subgraphs of $G$ isomorphic to 
            triangle graphs $\Delta_m$  with $m \leq n$, where
            $m$ and $n$ have the same parity. 
        \item Its edges are defined as follows:
            \begin{enumerate}
                \item Two subgraphs (of $G$) $S_1 \cong \Delta_m$ and $S_2 \cong \Delta_m$
                    are adjacent (in $G_n$) if $S_1 \subset \neig{G}{S_2}$ or 
                    $S_2 \subset \neig{G}{S_1}$.\footnote{These two conditions are in fact 
                    equivalent. This is a direct consequence of Lemma 
                    \ref{Lem_HexagonalChartExtension}, shown later.}
                \item Two subgraphs $S_1 \cong \Delta_m$ and $S_2 \cong \Delta_{m-2}$
                    (with $m \geq 2$) are adjacent if $S_2 \subset S_1$.
                \item Two subgraphs $S_1 \cong \Delta_m$ and $S_2 \cong \Delta_{m-4}$
                    (with $m \geq 4$) are adjacent if $S_2 \subset S_1$ and
                    $S_2$ does not contain any vertex $\partial S_1$, i.\,e. $S_2\cap \partial S_1=\emptyset$.
                \item Two subgraphs $S_1 \cong \Delta_m$ and $S_2 \cong \Delta_{m-6}$
                    (with $m \geq 6$) are adjacent if $S_2 \subset S_1$ and
                    $S_2$ does not contain any vertex with distance at most 1 
                    from the boundary of $\partial S_1$, i.\,e. $S_2\cap N_G[\partial S_1]=\emptyset.$
            \end{enumerate}
    \end{itemize}
A subgraph $S\cong \Delta_m$ 
of $G$ is said to be of \emph{level} $\mathemph{m}$ in $G_n$.
\end{defi}	

\noindent Clearly, $G_0 = G$. Thus we can try to prove $G_n=k^nG$ by induction.

\begin{ex}\label{Ex_CliqueGraphAdjacency_four}
	The subgraph $\Delta_4$ of the \mammal\enspace $\Hex_4$ is a vertex of every geometric clique graph $(\Hex_4)_n$ with an even $n\geq 4$. The adjacent vertices of level $0$ are the three $\Delta_0$ that are depicted in blue in Figure \ref{Fig_DeltaFour} and the adjacent vertices of level $2$ are the ``face-down`` $\Delta_2$, which is depicted in red, and the six ``face-up'' $\Delta_2$, two of which are depicted in yellow.

%    For a subgraph $\Delta_4\cong S\subseteq G$,
%    regarded as a vertex of $G_n$,
%     we describe the adjacencies to vertices of lower 
%    level. The subgraph $S$ is adjacent to six ``face-up'' $\Delta_2$, of which
%    two are drawn yellow in Figure \ref{Fig_DeltaFour}. There is another one
%    ``face-down`` $\Delta_2$, drawn in red. Finally, there are three adjacent
%    $\Delta_0$, marked with blue.
    \begin{figure}[htbp]
    	\centering
        \begin{tikzpicture}[scale=0.5]
        \HexagonalCoordinates{4}{4}
        
        \foreach \a/\b/\c in {00/20/02, 11/31/13}{
                \fill[yellow] (A\a) -- (A\b) -- (A\c) -- cycle;
        }
        
        \draw
            (A00) -- (A40) -- (A04) -- (A00)
            (A01) -- (A31) -- (A30) -- (A03) -- (A13) -- (A10) -- (A01)
            (A02) -- (A20) -- (A22) -- (A02);
        
        \draw[very thick, red] (A20) -- (A22) -- (A02) -- cycle;
        
        \foreach \p in {11, 21, 12}{
            \fill[blue] (A\p) circle (4pt);
        }
        \end{tikzpicture}
        \caption{The (types of) subgraphs of level $0$ and $2$ of $\Hex_4$ that are adjacent to $\Delta_4$ in any $(\Hex_4)_n$ with an even $n\geq4$}
        \label{Fig_DeltaFour}
    \end{figure}
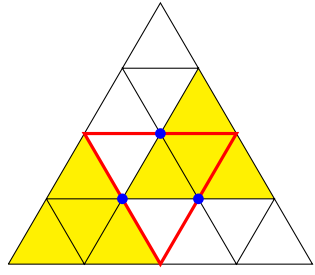

\end{ex}

\subsection{Properties of Triangles}\label{Subsect_Triangles}
In this subsection, we collect some technical results about triangles
and their relations.

\begin{rem}\label{Rem_TriangleInclusionBoundaryDistance}
    Let $m \geq 3$.
    The vertices of $\Delta_m$ with distance at least 1 to the boundary 
    induce the graph 
    $\Delta_m\setminus \partial \Delta_m\cong\Delta_{m-3}$. Thus, for 
    $m \geq 6$, the vertices with
    distance at least 2 induce 
    $\Delta_m\setminus N_G[\partial \Delta_m] \cong\Delta_{m-6}$.
\end{rem}

\noindent We define some graphs and vertex sets for future reference.
     The set
    \begin{align*}
	\mathemph{\oldvec{E}}&\isdef V_1\cap\Z_{\geq 0}^3=\{(1,0,0),(0,1,0),(0,0,1)\}\intertext{is the canonical basis, the graph}
	\mathemph{\nabla_{1}}&\isdef\Hex_2[(1,1,0),(0,1,1),(1,0,1)]\intertext{is the downward triangle of side length 1 in the centre of $\Delta_2$, the graphs}
        \mathemph{\nabla_{1}^{\oldvec{e}}}&\isdef\Hex_3[(1,1,0)+\vec{e},(0,1,1)+\vec{e},(1,0,1)+\vec{e}]\intertext{with $\vec{e}\in\vec{E}$ are the downward triangles of side length 1 inside $\Delta_3$, and the graph} 
	\mathemph{\nabla_{2}}&\isdef\Hex_4[(2,2,0),(0,2,2),(2,0,2)]
    \end{align*}
    is the downward triangle of side length 2 in the centre of $\Delta_4$.

The following two auxiliary lemmas discuss small special cases. 

\begin{lem}\label{Lem_TriangleInclusionOne}
    Let $m \geq 1$ and consider $\Delta_m \subset \Hex_m$. If
    $\Delta_{m-1} \cong S \subset \Delta_m$, either
    \begin{enumerate}
	\item $S$ is the image of $\Delta_{m-1}^{\vec{e}}$ with $\vec{e} \in \vec{E}$, or
	\item $m = 2$ and $S=\nabla_1$.
    \end{enumerate}
    In particular, $\Delta_m \subset \neig{G}{S}$.
\end{lem}
\begin{proof}
	See Appendix \ref{Sect_propoftriang}
\end{proof}

\begin{lem}\label{Lem_TriangleInclusionTwo}
    Consider $\Delta_m \subset\Hex_m$ with $m\geq 2$.
    If $\Delta_{m-2} \cong S \subset \Delta_m$, either
    \begin{enumerate}
	\item $S$ is the image of $\Delta_{m-2}^{\vec{f}}$ with $\vec{f} \in \vec{E}+\vec{E}=V_2\cap\Z_{\geq 0}^3$, 
	\item $m=3$ and $S=\nabla_{1}^{\vec{e}}$ for
	    some $\vec{e} \in \vec{E}$, or
	\item $m=4$ and $S=\nabla_{2}$.
    \end{enumerate}
\end{lem}
\begin{proof}
	See Appendix \ref{Sect_propoftriang}
\end{proof}

\subsection{Clique construction of $G_n$}\label{Subsect_CliqueConstruction}
In this subsection, we construct different cliques of $G_n$. The constructed 
cliques fall into two classes; those that are formed from three $\Delta_m$ within
one $\Delta_{m+1}$ (Lemma \ref{Lem_CliqueConstructionTriangle}), 
and those that are formed by all $\Delta_1$ incident to a vertex
(Lemma \ref{Lem_CliqueConstructionVertex}).

In the the next lemma, we employ a shorthand: For a hexagonal chart
$\mu: \Delta_{m+1} \to S$ and $(t_1,t_2,t_3) \in \Z^3$, we denote
the image of $\mu \circ \Delta_{m+1-t_1-t_2-t_3}^{t_1,t_2,t_3}$ by
$\mathemph{\mu^{t_1,t_2,t_3}}$.

\begin{lem}\label{Lem_CliqueConstructionTriangle}
    Let $G$ be a \mammal\enspace and 
    $\Delta_{m+1} \iso{\mu} S \subset G$ a hexagonal chart with
    $m \leq n$ and $m\equiv_2 n$. 
    The common
    neighbourhood $\cneig{G_n}{\mu^{1,0,0},\mu^{0,1,0},\mu^{0,0,1}}$
    forms a clique in $G_n$.
\end{lem}
\begin{proof}
    By Definition \ref{Def_theCliqueGraph}, the $\mu^{\vec{e}}$ with $\vec{e}\in\vec{E}$
    are vertices of $G_n$.
    They are all contained in $S \subset \neig{G}{\mu^{\vec{e}}}$ 
    (by Lemma \ref{Lem_TriangleInclusionOne}).
    Thus, by Definition \ref{Def_theCliqueGraph}, 
    they are all adjacent to each other.
    The first step of the proof is finding 
    all elements in the common neighbourhood.
    Let $T$ be in $\cneig{G_n}{\mu^{1,0,0},\mu^{0,1,0},\mu^{0,0,1}}\setminus\{\mu^{1,0,0},\mu^{0,1,0},\mu^{0,0,1}\}$.

    \begin{enumerate}
        \item If $T \cong \Delta_{m-k}$ for $k \in \{2,4,6\}$ and $k\leq m$, 
            by Definition \ref{Def_theCliqueGraph}, 
            $T \subset\mu^{1,0,0}\cap \mu^{0,1,0}\cap \mu^{0,0,1}$. 
            For $m\in \{0,1\}$ we have 
            $\mu^{1,0,0}\cap \mu^{0,1,0}\cap \mu^{0,0,1}=\emptyset$, 
            thus this is a contradiction.
            For $m \geq 2$, we have 
            $\mu^{1,0,0}\cap \mu^{0,1,0}\cap \mu^{0,0,1}=\mu^{1,1,1} \cong \Delta_{m-2}$. 
            We distinguish between the possible values of $k$.
            \begin{enumerate}
            	\item  $k=2$: We conclude $T = \mu^{1,1,1}$.
            	\item $k=4$: We have 
            	$T\subseteq \mu^{\vec{e}}\setminus\partial \mu^{\vec{e}}$ 
            	for the three $\vec{e}\in\vec{E}$. Thus, by Remark 
            	\ref{Rem_TriangleInclusionBoundaryDistance},\\ 
            	$\Delta_{m-4}\cong T\subseteq \mu^{1,1,1}\setminus\partial\mu^{1,1,1}\cong \Delta_{m-5}$, 
            	which is impossible.
            	\item $k=6$: We have 
            	$T\subseteq \mu^{\vec{e}}\setminus\partial \neig{G}{\mu^{\vec{e}}}$ 
            	for the three $\vec{e}\in\vec{E}$. Thus, by Remark 
            	\ref{Rem_TriangleInclusionBoundaryDistance}, 
            	$\Delta_{m-6}\cong T\subseteq \mu^{1,1,1}\setminus\partial\neig{G}{\mu^{1,1,1}}\cong \Delta_{m-8}$, 
            	which is impossible.
            \end{enumerate}
	    
            We  conclude $m\geq k=2$ and $T = \mu^{1,0,0}\cap \mu^{0,1,0}\cap \mu^{0,0,1}=\mu^{1,1,1}$.
        \item If $T \cong \Delta_m$, by Definition \ref{Def_theCliqueGraph}, 
            $T \subset \neig{G}{\mu^{1,0,0}}\cap \neig{G}{\mu^{0,1,0}}\cap \neig{G}{\mu^{0,0,1}}$. 
            Since by  Lemma \ref{Lem_NeighbourhoodIntersection}, 
            $S=\neig{G}{\mu^{1,0,0}}\cap \neig{G}{\mu^{0,1,0}}\cap \neig{G}{\mu^{0,0,1}}$, 
            Lemma \ref{Lem_TriangleInclusionOne} shows that $T$ can 
            only appear if $m = 1$ (in which case it comes from a $\nabla_1$).
            
            In particular, it is never part of the common neighbourhood
            if $\mu^{1,1,1}$ is.
        \item If $T \cong \Delta_{m+k}$ for $k \in \{2,4,6\}$, we have
            $\Delta_{m+1}\cong S = \mu^{1,0,0}\cup\mu^{0,1,0}\cup\mu^{0,0,1}\subset T$.
             Again, we distinguish between the possible values of $k$.
            \begin{enumerate}
            	\item  $k=2$:  We employ Lemma \ref{Lem_TriangleInclusionOne} 
            	to describe
            	how $S$ can lie in $T$. By the same lemma and 
            	Definition \ref{Def_theCliqueGraph},
            	all of these $T$ are pairwise adjacent.
            	
            	Consider adjacency to smaller triangles: 
            	If $m=1$, the additional $\Delta_m$ from Lemma \ref{Lem_TriangleInclusionOne} 
            	lies in $S$ and is thus adjacent to all $T$.
            	If $m \geq 2$, the intersection
            	$\mu^{1,1,1} \cong \Delta_{m-2}$ has distance 1
            	to $\partial S$. Thus, it also
            	has distance 1 to $\partial T$ and it is therefore
            	adjacent to all of them.
            	\item  $k=4$: The subgraphs $\mu^{\vec{e}}$ need to have distance 1 to
            	the boundary of $T$. Thus,
            	$S$ also has this distance. By
            	Remark \ref{Rem_TriangleInclusionBoundaryDistance},
            	this uniquely defines $T$  with $\neig{G}{S} \subset T$.
            	
            	Since the additional $\Delta_{m+2}$ lie in $\neig{G}{S}$
            	(Lemma \ref{Lem_TriangleInclusionOne}),
            	they are adjacent to $T$. For $m=1$, the additional
            	$\Delta_m$ lies in
            	$S$ and has distance 1 from the boundary of $T$.
            	For $m \geq 2$, the intersection
            	$\mu^{1,1,1}$ has distance 1 from $\partial S$.
            	Since $S$ has distance 1 from
            	$\partial T$, the total distance between
            	$\mu^{1,1,1}$ and $\partial T$ is 2, showing
            	their adjacency.
            	\item  $k=6$: By Remark 
            	\ref{Rem_TriangleInclusionBoundaryDistance}, there is only
            	one embedding $\Delta_m \to \Delta_{m+6}$ with distance 2 
            	to the boundary. Thus,
            	there is no such element adjacent to 
            	all $\mu^{\vec{e}}$ simultaneously. 
            \end{enumerate}
    \end{enumerate}
Finally, we conclude that $\cneig{G_n}{\mu^{1,0,0},\mu^{0,1,0},\mu^{0,0,1}}$ 
is a clique of $G_n$.\qedhere
\end{proof}

\noindent After having covered the triangle case, we now cover the vertex case.
The neighbours of a vertex $v$ form a circle $w_1w_2\dots w_k$ for
some $k\in \N$. 
    The \emph{umbrella} of $v$ is the set containing the $k$
    facets $\{v,w_i,w_{i+1}\}$ for all $1 \leq i \leq k$, where
    we read indices modulo $k$.

\begin{lem}\label{Lem_CliqueConstructionVertex}
    Let $G$ be a \mammal\enspace and $v$ a vertex in $G$.
    For odd $n$, the common neighbourhood $\cneig{G_n}{T\cong \Delta_1\mid v\subseteq \Delta_1}$ of all
    $\Delta_1$ containing $v$ forms a clique in $G_n$.
\end{lem}
\begin{proof}
    Clearly, all these $\Delta_1$ are pairwise adjacent in $G_n$, since they share $v$. Thus, they lie in a clique, which itself lies in the common neighbourhood  $\cneig{G_n}{T\cong \Delta_1\mid v\subseteq \Delta_1}$. \\
    We consider all $\Delta_{1+k} \cong T$ which lie in this common neighbourhood (for
    $k \in \{0,2,4,6\}$). 
    \begin{enumerate}
    	\item Case $k=0$: If there was a $\Delta_1$ adjacent to all facets in the
    	umbrella, all of its vertices would lie in $N_G(v)$ (each
    	of its vertices can only lie in two facets and the number of facets
    	is at least 6). Thus, $N_G(v)$ contains a \thcir, in contradiction
    	to being at least a 6-cycle.
    	\item Case $k=2$: Any $\Delta_3$ which is  adjacent to all facets in
    	the umbrella, would be the $\Delta_3$ containing $v$ as its middle vertex. Thus, $v$ has degree $6$. 
    	In this case, there are two 
    	$\Delta_3$ with $v$ as their central vertex and these two
    	are clearly adjacent.
    	\item Case $k=4$: By Remark \ref{Rem_TriangleInclusionBoundaryDistance}, 
    	every $\Delta_1$
    	adjacent to a given $\Delta_5$ has to lie within the central $\Delta_2$. 
    	By Lemma
    	\ref{Lem_TriangleInclusionOne}, there are four of 
    	these $\Delta_1$. 
    	Since $\deg(v) \geq 6$, no $\Delta_5$ can be adjacent to 
    	all the facets in the umbrella.
    	\item Case $k=6$:  By Remark \ref{Rem_TriangleInclusionBoundaryDistance} 
    	and Definition \ref{Def_theCliqueGraph}, a $\Delta_7$ is only 
    	adjacent to one $\Delta_1$.
    	Since $\deg(v) \geq 6$, no $\Delta_7$ can be adjacent to all 
    	the facets in the umbrella.
    \end{enumerate}
   
    \noindent Thus, all the elements in 
    $\cneig{G_n}{T\cong \Delta_1\mid v\subseteq \Delta_1}$ are 
    pairwise adjacent, and we obtain a clique.
\end{proof}

\noindent Those two lemmas suggest a correspondence between the cliques of $G_n$ we constructed and the vertices of $G_{n+1}$. 
 
\begin{rem}\label{lem_map}
    For every \mammal\enspace $G$, and every $n\in \Z_{\geq 0}$ there is a map
    \begin{align*}
        C\colon V(G_{n+1})&\to \{\text{cliques of }G_n\},\\
        S&\mapsto\text{the clique from}
        \begin{cases}
            \text{Lemma \ref{Lem_CliqueConstructionVertex}}, &\text{if $S$ is of level $0$},\\
            \text{Lemma \ref{Lem_CliqueConstructionTriangle}}, &\text{otherwise.} 
	\end{cases}
    \end{align*}
\end{rem}

\noindent In Section \ref{Sect_ChartExtension}, we discuss some theory 
that helps to show the bijectivity of this map in Section \ref{Sect_CliqueCompleteness}.
In Section \ref{Sect_CliqueIntersections}, we prove that it is a 
graph isomorphism between $G_{n+1}$ and $kG_{n}$.

\section{Chart Extensions}\label{Sect_ChartExtension}
In Section \ref{Sect_Graph}, we introduced the graph $G_n$ 
(of a \mammal\enspace $G$) and constructed
several of its cliques in Subsection \ref{Subsect_CliqueConstruction}. We
still need to show that $G_n$ has no more cliques than those.

To do so, we transfer local regions of $G_n$ to local regions of the 
hexagonal grid, where the calculations become simpler. This transfer is
easy if we only consider ``smaller`` triangles within a ``larger`` hexagonal
chart. However, Definition \ref{Def_theCliqueGraph} also includes edges
between triangles of the same size. In this case, we
extend a hexagonal chart to a larger domain, containing
all adjacent triangular-shaped subgraphs as well, if possible.
The existence of such an extension is non-trivial and
needs several intricate arguments about topology and straight paths.
In this technical chapter, we show that such an extension of charts is always
possible for $m \geq 3$ (Lemma \ref{Lem_HexagonalChartExtension}).

For $m\geq 4$ and a $\Delta_m$-shaped subgraph $S$ of 
$\Hex_m$, any neighbouring $\Delta_m$-shaped subgraph 
(with respect to the geometric clique graph $(\Hex_m)_n$) can be constructed
by adding vectors from $\vec{D}_0$ (see Def. \ref{Def_HexagonalGrid}) to $S$. 
For $m=3$, one additional neighbour occurs, which is the rotation of $S$ by 
$\frac{\pi}{3}$.  We want to prove that the same structure holds for 
subgraphs of $G$. We start by noting that
each $\vec{d} \in \vec{D}_0$ can lead to an adjacent $\Delta_m$.
(The complementary claim is proven in Lemma \ref{Lem_NeighbourChart}).

\begin{rem}\label{Rem_TranslationNeighboursExist}
	Let $G$ be a \mammal\enspace and let $\nu\colon H\to F$ be a 
        hexagonal chart, with $\Delta_m\subseteq H\subseteq \Hex_m$ for an $m\geq 3$. 
	If for a $\vec{d}\in \vec{D}_0$, the image of 
	 $\Delta_m^{\vec{d}}$ is
	contained in $H$, the image of $\nu\circ\Delta_m^{\vec{d}}$
	lies inside $\neig{G}{\nu(\Delta_m)}.$
\end{rem}
%\begin{proof}
%	By construction of $\vec{D}_0$,
%	the image of $\Delta_m^{\vec{d}}$ lies in $\neig{G}{\Delta_m}$. Isomorphisms
%	preserve neighbourhood, so this is equivalent to
%	$\nu(\Delta_m^{\vec{d}}(\Delta_m)) \subset \neig{G}{\nu(\Delta_m)}= \neig{G}{S}$.
%\end{proof}	

\noindent We employ \emph{facet-paths} to extend charts. For a
locally cyclic graph with boundary $G = (V,E)$, these are finite
sequences of facets
    $f_1f_2\dots f_k$ such that $f_i \cap f_{i+1} \in E$
    for all $1 \leq i < k$.

Given a monomorphism $\mu: H \to G$ and a facet-path
$f_1 f_2 \dots f_k$ in $H$, we have the following transfer:
For each pair $f_if_{i+1}$ with
$1 \leq i < k$, the image $\mu(f_i \cap f_{i+1})$ is an 
inner edge and $\mu(f_i) \neq \mu(f_{i+1})$.
In particular, the images of the vertices in
$f_i$ uniquely determine the image of the
vertex $f_{i+1}\ohne(f_i \cap f_{i+1})$.
Inductively, the images of the vertices in $f_1$
determine those in $f_k$.

This allows a unique extension along a facet-path.
Unfortunately, the extensions from different facet-paths
are not compatible in general.

\noindent Given a $\Delta_m$-shaped subgraph $S\subseteq G$ with neighbouring $\Delta_m$-shaped subgraphs $T_k$, we want
to show the existence of a chart containing all of them. We start by
extending the chart from $S$ to one neighbouring
triangular-shaped subgraph. 

\begin{lem}\label{Lem_NeighbourChart}
    Let $G$ be a \mammal\enspace and $\Delta_m \iso{\mu} S \subset G$ be a standard
    chart with $m \geq 3$. Let $\Delta_m \cong T \subset \neig{G}{S}$. Then, there is a 
    $H \subset \Hex_m$ and a hexagonal chart $\hat{\mu}: H \to T$
    such that $\mu^{-1}(x) = \hat{\mu}^{-1}(x)$ holds for every vertex
    $x$ in $S \cap T$.
    Furthermore, $H$ falls in one of these two cases:
    \begin{itemize}
        \item $H = \vec{d} + \Delta_m$ for some $\vec{d} \in \vec{D}_0$.
        \item $H = \mathemph{\nabla_3} \isdef \{(a,b,c)\in\Hex_3 \mid a \leq 2, b \leq 2, c\leq 2\}$,
            with corner vertices $(2,2,-1)$, $(2,-1,2)$, and $(-1,2,2)$.
    \end{itemize}

    \noindent Consequently, for $m \geq 4$ there are at most six of these
     $\Delta_m\cong T \subset \neig{G}{S}$,
    and for $m=3$ there are at most seven.
\end{lem}
\begin{proof}
	By Remark \ref{Rem_StraightPathsInTriangle}, the boundary of 
	$\Delta_m$ consists of three straight paths of length $m$.
	Thus, to find $\Delta_m \cong T \subset \neig{G}{S}$,
	we start by describing all straight paths with $m$ edges within $\neig{G}{S}$. 
	Each of those paths either completely lies in  
	$\neig{G}{S}\ohne S$ or it intersects $S$ in at least one vertex.
	
	If all boundary paths of $T$ lay in $\neig{G}{S}\ohne S$, 
	Lemma \ref{Lem_NeighbourhoodCycle}
	would imply $\partial T = \neig{G}{S} \ohne S$ since both are cyclic graphs. 
	However, this would
	imply $S \subsetneq T$, contradicting $S \cong T$. 
	Thus, at least one boundary path of $T$ 
	intersects $S$. We can construct each of those
	paths by extending a straight path from $S$ into
	$\neig{G}{S}\ohne S$. The first extended edge cannot be a boundary
	edge of $\neig{G}{S}$.
	
	By Lemma \ref{Lem_NeighbourhoodCycle}, 
	vertices in $\neig{G}{S}\ohne S$ are incident to at most three facets.
	Thus, any straight path in
	$S$ can only be extended by one edge into $\neig{G}{S}$ 
	on each side (the path degree
	of any extension is at most 2), which can be seen in Figure \ref{Extensions_of_a_straight_path}.
	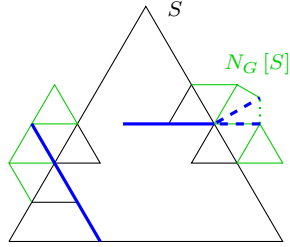
\begin{figure}[htbp]
		\centering
		\begin{tikzpicture} [scale=0.3]
		\HexagonalCoordinates{8}{8}
		\draw
		(A11) -- (A71) -- (A26) -- (A17) node[draw=none,fill=none,font=\scriptsize,midway,above right] {$S$} -- (A11)
		(A12) -- (A22) -- (A13) -- (A23) -- (A14);
		\draw[green!80!black]
		(A12) -- (A03) -- (A13) -- (A04) -- (A14) -- (A05) -- (A04) -- (A03);      
		\draw[very thick, blue] (A31) -- (A04);
		
		\draw (A53) -- (A43) -- (A44) -- (A34) -- (A35);
		\draw[green!80!black]
		(A54) -- (A63) -- (A53) -- (A54) -- (A44) -- (A45) -- (A35) node[draw=none,fill=none,font=\scriptsize,midway,above right] {$\neig{G}{S}$}	
		(A45) -- (D44);
		\draw[green!80!black,thick, dotted]
		(A54) -- (D44);
		
		\draw[very thick, blue] (A24) -- (A44);
		\draw[very thick,blue,dashed] (A54) -- (A44) -- (D44);			
		\end{tikzpicture}
		\caption{Possible extension of a straight path into the neighbourhood of $\Delta_{m}$.}
		\label{Extensions_of_a_straight_path}
	\end{figure}

	Consequently, we start with the straight paths in $S$ with at least $m-2$ 
	edges, extend them to straight paths of length $m$ and 
	add the possible two other sides of the $\Delta_m$-shaped subgraph.
	Up to symmetry, the preimages of the paths of length $m-2$ with respect to $\mu$ 
	are described in Remark \ref{Rem_StraightPathsInTriangle}, whose 
	notation ($\alpha$,
	$\beta$, and $\gamma$) we employ. 
	The images of $\alpha,\beta,\gamma \subseteq \Delta_m$ under 
	$\mu$ are called $\alpha^\mu, \beta^\mu,\gamma^\mu \subseteq S$.
	
	A straight path with $m$ edges can only be the boundary
	of at most two $T \cong \Delta_m$. However, if we start at a vertex
	$\mu(m-t-k,t,k) \in S$ and follow a straight path ``down'' 
	(i.\,e. in the direction
	of smaller third coordinates), it can only go $k$ steps while staying
	in $S$. Conversely, if we follow a straight path in the direction 
	of larger
	third coordinates (``up''), we can go at most $m-k$ steps within $S$.
	We conclude that from any vertex of 
	\begin{align*}
	\left\{\begin{array}{c}
	\alpha^\mu\\
	\beta^\mu\\
	\gamma^\mu
	\end{array}\right\} \text{ we can go}  \left\{\begin{array}{c}
	m+1\\
	m\\
	m-1
	\end{array}\right\}\text{ steps `up'}
	\text{ and }
	\left\{\begin{array}{c}
	1\\
	2\\
	3
	\end{array}\right\}
            \text{ steps `down'}
	\end{align*}
        staying in $\neig{G}{S}$.

	Since $m\geq 3$,  beginning at a vertex 
	of $\alpha^\mu$ or $\beta^\mu$ 
	only in the `up'-direction we find a straight path of 
	length $m$. For $m=3$, beginning at a vertex 
	of $\gamma^\mu$ only in the `down'-direction we have 
	the space for a straight path of length $m$ and for $m\geq 4$, no such straight path  beginning at a vertex 
	of  $\gamma^\mu$
	exists, neither `upwards' nor `downwards'.
	
	Now, we discuss which parts of $\alpha^\mu,\ \beta^\mu$, and $\gamma^\mu$ 
	can be the straight boundary paths of a 
	$\Delta_m \cong T \subset \neig{G}{S}$. Since $\alpha^\mu$ lies in 
	$\partial S$ with $S \cong \Delta_m$, it cannot
	be part of the boundary of $T$. Thus, two possible
	extensions of $\alpha$ and two possible extensions of $\beta^\mu$
        remain, as can 
	be seen in Figure \ref{Fig_paths}.
	
	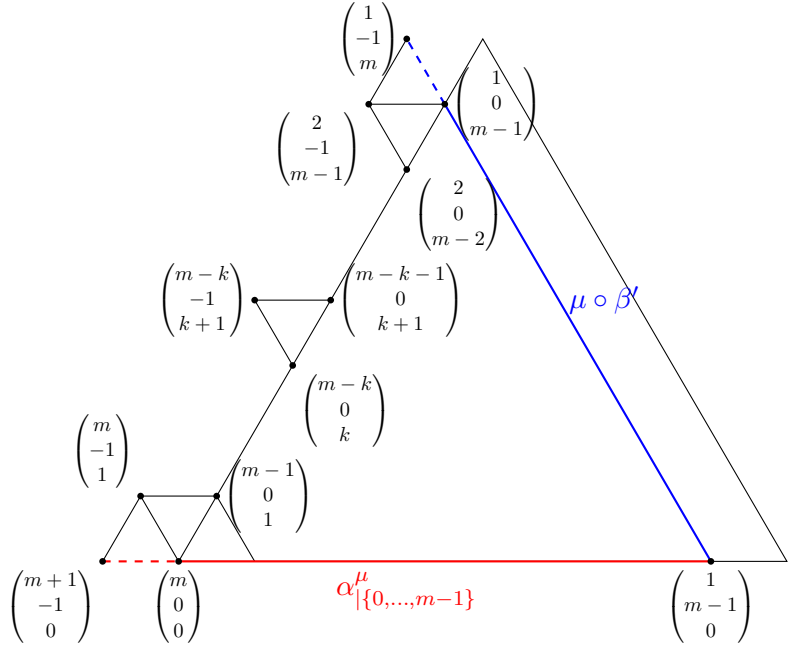
\begin{figure}[htbp]
		\begin{center}
			\begin{tikzpicture}[scale=0.5]
			\HexagonalCoordinates{9}{8}
			
			\draw (A10) -- (A90) -- (A18) -- cycle;
			\draw (A00) -- (A01) -- (A11) -- (A20);
			\draw (A01) -- (A10);
			\draw (A13) -- (A04) -- (A14);
			\draw (A08) -- (A07) -- (A16);
			\draw (A07) -- (A17);
			
			\draw[thick, red] (A10) -- (A80);
			\draw[thick, red, dashed] (A00) -- (A10);
			\node[red,below] at (A40) {$\alpha^\mu_{|\{0,\dots,m-1\}}$};
			\draw[thick, blue] (A80) -- (A17);
			\draw[thick, blue, dashed] (A17) -- (A08);
			\node[blue, right] at (A44) {$\mu\circ\beta'$};
			
			\foreach \p/\r/\x/\y/\z in {00/below left/m+1/-1/0, 10/below/m/0/0, 80/below/1/m-1/0,
				01/above left/m/-1/1, 11/right/m-1/0/1, 13/below right/m-k/0/k, 04/left/m-k/-1/k+1,
				14/right/m-k-1/0/k+1, 16/below right/2/0/m-2, 17/right/1/0/m-1, 07/below left/2/-1/m-1,
				08/left/1/-1/m}{
				\fill[black] (A\p) circle (2.5pt);
				\node[\r,scale=0.7] at (A\p) {$\begin{pmatrix}\x\\\y\\\z\end{pmatrix}$}; % Column format
				%\node[\r,scale=0.8] at (A\p) {$(\x,\y,\z)$}; % Row format
			}
			\end{tikzpicture}
		\end{center}
		\caption{Straight paths in a triangle}\label{Fig_paths}
	\end{figure}
	\noindent We extend $\alpha^\mu_{|\{0,\dots,m-1\}}$ at 0. This path ends at 
	$(1,m-1,0)$, where a rotation of $\beta^\mu$ starts:
	\begin{equation*}
	\beta': \{0,\dots,m-1\} \to \Z^3,\qquad t \mapsto (0,m-t,t).
	\end{equation*}
	If these paths lie in the boundary of $T$, the group action 
        from Subsection \ref{Subsect_Hexagonal}
        allows us to find a
	hexagonal chart $\nu: \Delta_m \to T$ with
	\begin{align*}
	\nu(0,m,0) = \mu(1,m-1,0) \qquad\text{and}\qquad \nu(1,m-1,0) = \mu(2,m-2,0).
	\end{align*}
	A translation allows us to rewrite the hexagonal chart as
	\begin{align*}
	\mu_{1,-1,0}: \Delta_m + (1,-1,0) \to T,\qquad  x \mapsto \nu(x+(-1,1,0)).
	\end{align*}
	In particular, all $\mu(k,0,m-k)$ with $2 \leq k \leq m-1$ have degree
	6 if this chart exists.
	
	Since the group action of Subsection \ref{Subsect_Hexagonal}
	acts transitively on the extensions of $\alpha$ and $\beta$,
	respectively, there are exactly six  triangular-shaped graphs $T$ that could be constructed
	in such a manner. These correspond to the elements of $\vec{D}_0$.
	
	To complete the construction of $\Delta_m \cong T \subset \neig{G}{S}$, we
	consider the path $\gamma$ for $m=3$. It has to be extended in both directions.
	Similarly to our construction of $\mu_{1,-1,0}$, we extend the chart $\mu$
	to incorporate the vertices $(2,2,-1)$, $(2,-1,2)$, and $(-1,2,2)$.
\end{proof}

\noindent Next, we combine these different hexagonal charts. We show that they
define compatible maps.
\begin{lem}\label{Lem_ChartExtensionCompatible}
    Let $G$ be a \mammal\enspace and 
    $\mu: \Delta_m \to S \subset G$ be a standard chart with $m \geq 3$.
    Let $\mu_1: H_1 \to T_1$ and $\mu_2: H_2 \to T_2$ be two hexagonal charts
    from Lemma \ref{Lem_NeighbourChart}. For any $x \in H_1 \cap H_2$, we have
    $\mu_1(x) = \mu_2(x)$.
\end{lem}
\begin{proof}
	Let $x \in H_1 \cap H_2$. If $x \in \Delta_m$, the claim follows directly
	from Lemma \ref{Lem_NeighbourChart}. Otherwise, we have to consider the
        extension construction along facet-paths.
	
	If there is a facet path $f_1f_2$ in $H_1\cap H_2$ with 
	$f_1 \subset \Delta_m$ and
	$x \in f_2$, both $\mu_1$ and $\mu_2$ have to map $x$ to the same value. 
	Thus, only the
	corner vertices of $\vec{d} + \Delta_m$ might be problematic.
	
	Without loss of generality (Subsection \ref{Subsect_Hexagonal}),
	let $H_1 = (1,-1,0) + \Delta_m$. The corner vertex $(m+1,-1,0)$ does not lie
	in any $H_2$, so it can be ignored. The corner $x = (1,-1,m)$ also lies
	in $H_2 = (0,-1,1) + \Delta_m$. We can define $x$ by a facet path in $H_1$ 
	with three facets. Since $m\geq 3$, this facet path also lies in $H_2$. 
	Thus, the charts
	$\mu_1$ and $\mu_2$ cannot conflict, as can be seen in Figure 
	\ref{Fig_ExtendingtheStandardChart}.
	\begin{figure}[htbp]
		\begin{center}
			\begin{tikzpicture}[scale=0.3]
			\HexagonalCoordinates{9}{8}
			\draw[green!80!black] (A00) -- (A60) -- (A06) -- cycle;
			\draw (A10) -- (A70) -- (A16) -- cycle;
			\draw[purple] (A01) -- (A61) -- (A07) -- cycle;
			
			\draw[blue] (A12) -- (A03) -- (A13) -- (A22) -- cycle;
			\draw[blue] (A12) -- (A13);
			
			\draw[blue] (A06) -- (A15) -- (A05) -- cycle;
			\draw[blue] (A14) -- (A15) -- (A05) -- cycle;
			\draw[blue] (A14) -- (A15) -- (A24) -- cycle;
			
			\foreach \p in {3,6,0}{
				\fill[blue] (A0\p) circle (5pt);
			}
			\end{tikzpicture}
			\caption{The edge-face-paths that can be used for extending the standard chart to a neighbouring $\Delta_m$}\label{Fig_ExtendingtheStandardChart}
			
		\end{center}
	\end{figure}
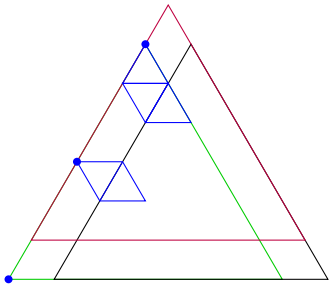
\end{proof}

\noindent Finally, we put all of the pieces together.

\begin{lem}\label{Lem_HexagonalChartExtension}
    Let $G$ be a \mammal\enspace and 
    $\mu: \Delta_m \to S \subset G$ be a standard chart
    with $m \geq 3$.
    There is a hexagonal chart $\hat{\mu}: E \to \hat{S}$, with $\Delta_m \subset E$
    and $S \subset \hat{S} \subset G$, such that $\hat{\mu}_{|\Delta_m} = \mu$
    and
    such that any $T \cong \Delta_m$ with $T \subset \neig{G}{S}$
    is either
    \begin{itemize}
        \item the image of $\hat{\mu} \circ \Delta_m^{\vec{t}}$ for some $\vec{t} \in \vec{D}_0$
            or
        \item the image of $\hat{\mu}(\nabla_3)$ from Lemma \ref{Lem_NeighbourChart} if $m=3$.
    \end{itemize}
\end{lem}
\begin{proof}
	For any $\Delta_m \cong T \subset \neig{G}{S}$, we construct the hexagonal chart
	$\mu_T : H_T \to T$ from Lemma \ref{Lem_NeighbourChart}. We define $E$ as
	the union of $\Delta_m$ with all those $H_T$.
	
	For $m > 3$, this graph is
	\begin{equation*}
	E \isdef \Delta_m + \bigcup_{\vec{d}\in \vec{D}} (\vec{d}+\Delta_m),
	\end{equation*}
	in which $\vec{D}$ is the set of all applicable translation vectors.
	For $m = 3$, the graph also contains $\nabla_3$. 
	By Lemma \ref{Lem_NeighbourChart} and Lemma \ref{Lem_ChartExtensionCompatible}, we can define $\hat{\mu}$ as follows:
	\begin{equation*}
	\hat{\mu}: E \to G \qquad x \mapsto 
	\begin{cases}
	\mu(x), & x \in \Delta_m, \\
	\mu_T(x), & x \in H_T \text{ for } \Delta_m \cong T \subset \neig{G}{S}.
	\end{cases}
	\end{equation*}
	It remains to show that $\hat{\mu}$ is injective.
	Since $\mu$ is
	injective on $S$, we only need to consider pairs of vertices, in which at least one is from $\neig{G}{S}\ohne S$.
	
	\begin{enumerate}
		\item Let $x \in E\ohne\Delta_m$ and $y \in \Delta_m$
		such that $\hat{\mu}(x) = \hat{\mu}(y) \in S$. 
		Then, there is a $\vec{d} \in \vec{D}_0$ such that
		$x \in \vec{d} + \Delta_m$. By construction of
		$\mu_T$, the point $x$ is mapped to a point in
		$\neig{G}{S}$, which is different from $S$ by
		Lemma \ref{lem_topology}(\ref{top_a}), in contradiction
		to our assumption.
		\item Let $x,y \in E\ohne\Delta_m$ such that
		$\hat{\mu}(x) = \hat{\mu}(y)$.
		Except the translates of the triangle tips, every
		vertex in $E\ohne\Delta_m$ in adjacent
		to two different boundary vertices of $\Delta_m$. Thus,
		we first consider the case where $x$ and $y$ are adjacent
		to different boundary vertices $a$ and $b$, respectively. Thus, we have
		a path of length 2 from $\hat{\mu}(a) \in \partial S$ over
		$\hat{\mu}(x) \in \neig{G}{S}\ohne S$ to 
		$\hat{\mu}(b) \in \partial S$. Then, Lemma \ref{lem_topology}\eqref{top_b}
		implies that $\{\hat{\mu}(a),\hat{\mu}(b),\hat{\mu}(x)\}$ is 
		a facet. By our proof of well-definedness, this is only possible
		if $x=y$.
		
		It remains to show the claim if $x$ and $y$ both are adjacent
		to the same triangle tip, say $x = (0,-1,m+1)$ and 
		$y = (-1,0,m+1)$. In this case, $\hat{\mu}(x) = \hat{\mu}(y)$
		would imply $\deg(\hat{\mu}(0,0,m)) = 5$, in contradiction
		to $\delta=6$.
		\qedhere
	\end{enumerate}
\end{proof}

    The condition $m\geq 3$ in Lemma \ref{Lem_HexagonalChartExtension} is 
    necessary. Extending a chart $\mu\colon \Delta_2\to S$ simultaneously 
    to all subgraphs $T\cong\Delta_{2}$ of $G$ is only possible if the 
    vertices $(1,1,0),(1,0,1)$, and $(0,1,1)$ are mapped to vertices of degree $6$. 

\section{Full Clique Description}\label{Sect_CliqueCompleteness}
In this section, we show that the cliques from Subsection
\ref{Subsect_CliqueConstruction} are all cliques of the geometric
clique graph $G_n$.
This section culminates in a full description of all cliques of $G_n$ (Corollary
\ref{Cor_CliqueSummary})
and the correspondence to the vertices of $G_{n+1}$ (Theorem \ref{Theo_bijective}).

For $m \geq 3$, we employ the chart extensions from Section
\ref{Sect_ChartExtension}. The smaller cases have to be argued
differently.

\subsection{Exceptional (Small) Cases}
In this subsection, we discuss the cliques which only contain elements of 
levels smaller than 3.
\begin{lem}\label{Lem_EvenCliquesOfSmallLevel}
    Let $C$ be a clique of $G_n$,
     %where every element is isomorphic to $\Delta_0$ or $\Delta_2$.
     in which every vertex is of level $0$ or $2$.
     Then, $C$ is one of the
    cliques described in Lemma \ref{Lem_CliqueConstructionTriangle}.
\end{lem}
\begin{proof}
    We start with the case where all vertices of $C$ are of level $0$, i.e. they are isomorphic
    to $\Delta_0$. In this case, they form a clique
    of $G$, i.\,e. a triangle $S\cong \Delta_1$. So, $C$ is
     constructed from $S$ by Lemma
    \ref{Lem_CliqueConstructionTriangle}. 

    For the remainder, we assume that $C$ contains a vertex of level $2$, i.e.\ a subgraph 
    $S\cong \Delta_{2}$ of $G$. Thus, $C$ lies in the closed neighbourhood 
    $N_{G_n}[S]$.
    We visualise %\footnote{In general, we cannot extend the chart of
    %$\Delta_2$ to include all elements in $C$.} 
    the neighbourhood in Figure \ref{Fig_mtwo}.
    Remark \ref{Rem_TranslationNeighboursExist} shows that all the depicted $\Delta_2$-shaped subgraphs exist.
    We label the subgraphs which are isomorphic to $\Delta_0$ with their preimage under a standard chart of $S$.
    Since it is not necessarily possible to extend this chart to all the 
    $\Delta_2$-shaped subgraphs in the neighbourhood, we label those in 
    a new labelling scheme. We place every label inside the central facet of the subgraph.
    Two different $\Delta_2$-shaped subgraphs
    are adjacent if and only if their central facets have facet-distance at
    most 2.

    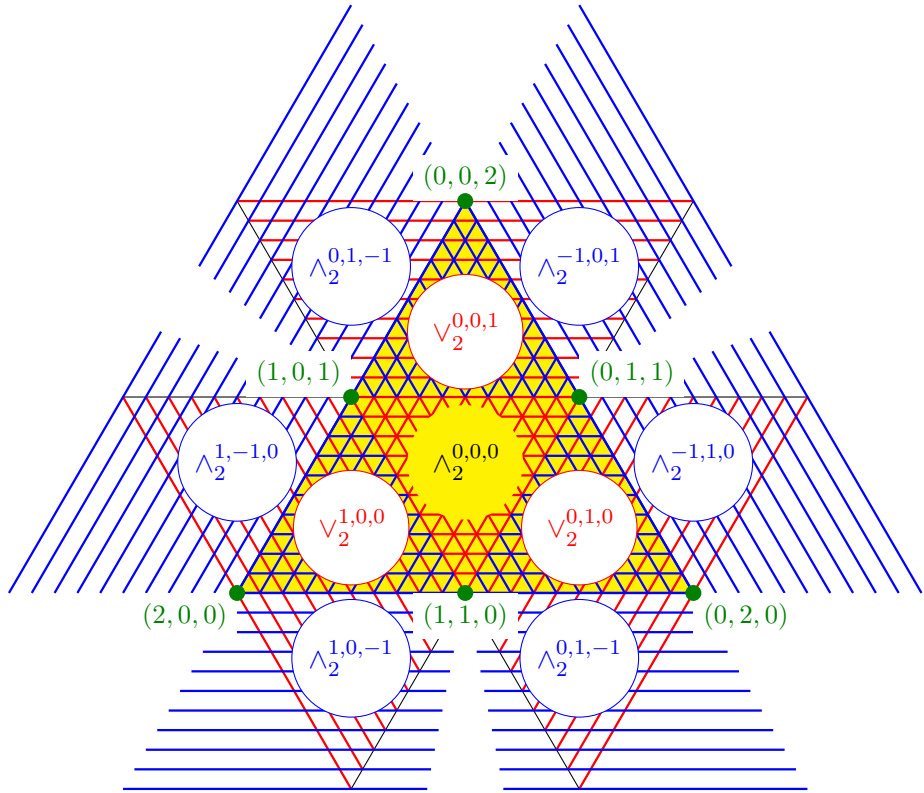
\begin{figure}[htbp]
	\centering
	\begin{tikzpicture}[scale=1.5]
	\HexagonalCoordinates{5}{5}
	
	\foreach \a\b in {20/30, 30/20, 02/03, 03/02, 32/23, 23/32}{
			\coordinate (T\a\b) at (barycentric cs:A\a=2/3,A\b=1/3);
	}
	
	\fill[yellow] (A11) -- (A31) -- (A13) -- cycle;
	\draw (A20) -- (A23) -- (A03) -- (A30) -- (A32) -- (A02) -- cycle;
	%\draw[ violet,line width=1mm] (A11) -- (A31) -- (A13) -- cycle;

	\foreach \a/\b/\c/\d/\e in {21/03/23/red/0.3, 12/30/32/red/0.3, 22/20/02/red/0.3}{
		\foreach \x in {0,0.05,0.1,0.15,0.2,0.25,0.3,0.35,0.4,0.45,0.5,0.55,0.6,0.65,0.7,0.75,0.8,0.85,0.9,0.95,1}{
			\draw[\d,line width=\e mm](barycentric cs:A\a=\x,A\b=1-\x)--(barycentric cs:A\a=\x,A\c=1-\x);
		}
	}
	\foreach \a/\b/\c/\d/\e in {12/13/22/blue/0.3, 22/12/13/blue/0.3, 21/11/12/blue/0.3,12/11/21/blue/0.3,21/31/22/blue/0.3,22/31/21/blue/0.3}{
		\foreach \x in {0,0.1,0.2,0.3,0.4,0.5,0.6,0.7,0.8,0.9,1}{
			\draw[\d,line width=\e mm](barycentric cs:A\a=\x,A\b=1-\x)--(barycentric cs:A\a=\x,A\c=1-\x);
		}
	}
\foreach \a/\b/\c/\d/\e/\f in {11/10/21/2030/blue/0.3,31/40/21/3020/blue/0.3, 11/01/12/0203/blue/0.3, 13/04/12/0302/blue/0.3, 13/14/22/2332/blue/0.3, 31/41/22/3223/blue/0.3}{
	\foreach \x in {0,0.1,0.2,0.3,0.4,0.5,0.6,0.7,0.8,0.9,1}{
		\draw[\e,line width=\f mm](barycentric cs:A\a=\x,A\b=1-\x)--(barycentric cs:A\c=\x,T\d=1-\x);
	}
}
	
	\foreach \a/\b/\c/\e in { 11/12/02/{1,-1,0}, 11/20/21/{1,0,-1}, 21/31/30/{0,1,-1}, 22/31/32/{-1,1,0}, 22/23/13/{-1,0,1}, 12/13/03/{0,1,-1}}{
		\node[shape=circle, draw=blue,fill=white] at (barycentric cs:A\a=1,A\b=1,A\c=1) {\textcolor{blue}{$\wedge_2^{\e}$}};
	} %12/21/22/{0,0,0},

%	\draw (A20) -- (A23) -- (A03) -- (A30) -- (A32) -- (A02) -- cycle;

	\node[shape=circle, fill=yellow] at (barycentric cs:A12=1,A21=1,A22=1) {\textcolor{black}{$\phantom.\wedge_2^{0,0,0\phantom.}$}};
	
	\foreach \a/\b/\c/\e in {11/21/12/{1,0,0}, 21/31/22/{0,1,0}, 12/22/13/{0,0,1}}{
		\node[shape=circle, draw=red,fill=white] at (barycentric cs:A\a=1,A\b=1,A\c=1) {\textcolor{red}{$\phantom.\vee_{2}^{\e\phantom.}$}};
	}
	
	\foreach \p/\r/\e in {11/below left/{2,0,0},21/below/{1,1,0},31/below right/{0,2,0}, 22/above right/{0,1,1}, 13/above/{0,0,2}, 12/above left/{1,0,1}}{
		
		\node[fill=white, \r] at (A\p) {\small\textcolor{green!50!black}{$(\e)$}}; 
		\fill[green!50!black] (A\p) circle (2pt);
	}
	
	\end{tikzpicture}
	\caption{Neighbourhood of an $S\cong\Delta_2$ in $G_n$, subgraphs isomorphic to $\Delta_2$ are labelled in their middle face with the symbols $\wedge$ for 'upward' and $\vee$ for 'downward' facing }\label{Fig_mtwo}
    \end{figure}

    \noindent We describe all the cliques of $N_{G_n}[S]$ which contain $S$ using the labels in Figure \ref{Fig_mtwo}.
    \begin{enumerate}
        \item If a corner-vertex of $S$, like $(2,0,0)$, 
            is contained in the clique, 
            the common neighbourhood of this vertex 
            and $S$ is a clique, which by Lemma 
            \ref{Lem_CliqueConstructionTriangle} is constructed from the
            $\Delta_1$ in $S$ containing the corner-vertex.
        \item Assume no corner-vertex of $S$ is contained in the clique.
            For all three corner-vertices, there must be an element in 
            the clique which is not adjacent to it; otherwise, the clique would not be maximal. From the remaining 
            elements in $N_{G_n}(S)$, the three middle vertices are each 
            adjacent to exactly two corner-vertices, the other elements 
            are each adjacent to exactly one corner-vertex. Thus, to 
            exclude the corner-vertices, either the three middle vertices 
            are in $C$ or at least one other element is in $C$.
            \begin{enumerate}
                \item\label{i} In the first case, the clique is constructed 
                    from the three middle-vertices using Lemma
                    \ref{Lem_CliqueConstructionTriangle}.
                \item\label{ii} In the second case, there is an element not 
                    adjacent to two of the corner vertices. Without loss of 
                    generality, these corner-vertices are $(0,2,0)$ and 
                    $(0,0,2)$ and the element not adjacent to them is  
                    $\wedge_2^{1,0,-1}$ or $\vee_{2}^{1,0,0},$  which will 
                    be called $S_1$. Additionally, in $C$ there needs to be
                    an element not adjacent to $(2,0,0)$ called $S_2$ which 
                    is adjacent to $S_1$ since they both lie in $C$.
                    
                    Thus, $S_2$ can be neither $\wedge_{2}^{-1,1,0}$ nor 
                    $\wedge_{2}^{-1,0,1}$ nor $\wedge_{2}^{0,-1,1}$ since 
                    those are not adjacent to each of the possible $S_1$. Picking 
                    $S_2$ to be $\vee_{2}^{0,1,0}$  is only possible if 
                    $S_1$ is $\vee_{2}^{1,0,0}$. In this case, $\cneig{G_n}{S,S_1,S_2}$ is a clique 
                    containing the middle-vertices and we are in case 
                    \ref{i}. The same happens if we choose $S_2$ to be 
                    $\vee_{2}^{0,0,1}$. 
                    
                    If the degree of  $(1,1,0)$ is at least $7$, there is no 
                    other possibility for $S_2$, but if the degree of 
                    $(1,1,0)$ is $6$, the vertices $\wedge_2^{1,0,-1}$ and 
                    $\wedge_2^{1,0,-1}$ are adjacent, and if 
                    $S_1=\wedge_2^{1,0,-1}$ we can choose $S_2$ to be 
                    $\wedge_2^{1,0,-1}$. 
		    In this case, $S$, $S_1$,
                    and $S_2$ are contained in a common $T\cong \Delta_3$,
                    from which $C$ is constructed by Lemma 
                    \ref{Lem_CliqueConstructionTriangle}.\qedhere
            \end{enumerate}
    \end{enumerate}
\end{proof}

\begin{lem}\label{Lem_OddCliquesofSmallLevel}
    Let $C$ be a clique of $G_n$,
    % where every element is isomorphic to $\Delta_1$. 
     in which every vertex is of level $1$.
     Then, $C$ is one of the
    cliques described in Lemma \ref{Lem_CliqueConstructionTriangle}
    or in Lemma \ref{Lem_CliqueConstructionVertex}.
\end{lem}
\begin{proof}
    If $C$ is not given as the common neighbourhood of 
    the set of facets incident to a given vertex 
    like in Lemma \ref{Lem_CliqueConstructionVertex}, 
    the intersection of the elements of $C$ is empty and 
    $C$ has at least three elements. Furthermore, there are two 
    elements of $C$ that do not intersect in an edge: otherwise, 
    for any three element subset of $C$ 
    there would be a vertex $v$ in the intersection of the three 
    elements and the neighbourhood of $v$ would contain a \thcir.
    
    Thus, we choose two elements $S$ and $T$ from $C$ which intersect 
    in a vertex $v$ but not in an edge. Since the common intersection 
    of all elements of $C$ is empty, there must be an element $U\in C$ 
    not containing $v$, but intersecting $S$ and $T$ in at least one 
    vertex each, which we will call $s$ and $t$. Those two vertices are 
    distinct since $S$, $T$, and $U$ do not have a common vertex, 
    and they are connected by a edge from $U$. As $s$ and $t$ also 
    lie in the neighbourhood of $v$, the edge $st$ also lies in this 
    neighbourhood. Since, by assumption, the third vertex of $U$ is 
    not $v$, it is the other common neighbour of $s$ and $t$. This 
    way, we proved that $C$ is constructed from the union of $S$, $T$, and $U$, which is $\Delta_2$-shaped, using
    Lemma \ref{Lem_CliqueConstructionTriangle}, as it is depicted in Figure \ref{Fig_triangclique}.
\end{proof}

      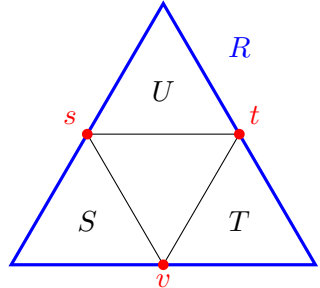
\begin{figure}[htbp]
    	\centering
    	\begin{tikzpicture}[scale=1]
    	\HexagonalCoordinates{5}{5}
    	
    	\draw (A21) -- (A12) -- (A22) -- cycle;
    	\draw[very thick, blue] (A11) -- (A31) -- (A13) -- cycle;
    	%\foreach \a/\b/\c/\e in {12/21/22/{0,0,0}, 11/12/02/{1,-1,0}, 11/20/21/{1,0,-1}, 21/31/30/{0,1,-1}, 22/31/32/{-1,1,0}, 22/23/13/{-1,0,1}, 12/13/03/{0,1,-1}}{
    		\node[blue] at (barycentric cs:A22=1,A13=1,A23=1) {$R$};
    	%}
    	
    	\foreach \a/\b/\c/\e in {11/21/12/S, 21/31/22/T, 12/22/13/U}{
    		\node at (barycentric cs:A\a=1,A\b=1,A\c=1) {$\e$};
    	}
    	
    	\foreach \p/\r/\e in {21/below/{v}, 22/above right/{t}, 12/above left/{s}}{
    		\fill[red] (A\p) circle (2pt);
    		\node[red, \r] at (A\p) {$\e$}; 
    	}
    	
    	\end{tikzpicture}
    	\caption{The clique of $G_n$ containing $S,T,$ and $U$ is constructed
          from their union $R\cong\Delta_2$ using Lemma 
          \ref{Lem_CliqueConstructionTriangle}. }\label{Fig_triangclique}
    \end{figure}

\subsection{The Generic (Large) Case}
Up to now we only investigated cliques lying in the lower levels of $G_n$. The cliques left to discuss are those containing a $\Delta_m$ with $m\geq 3$.
In this generic case, we describe the neighbourhood $\neig{G_n}{S}$ of a $S\cong \Delta_m$ explicitly
by using triangle inclusion maps. Then, we classify the cliques
there explicitly.

We can describe the adjacency conditions of Definition
\ref{Def_theCliqueGraph} combinatorially with triangle inclusion maps.
	Additional to the aforementioned set
	\begin{align*}
	&\vec{D}_{0\phantom{-}} \phantom{:}= \{ (1,-1,0), (1,0,-1), (-1,1,0), (0,1,-1), (-1,0,1), (0,-1,1) \}, \intertext{we define the following sets of coordinates:}
	&\mathemph{\oldvec{D}_{-2}} \isdef \{(2,0,0),(1,1,0),(0,2,0),(0,1,1),(0,0,2),(1,0,1)\}, \\
	&\mathemph{\oldvec{D}_{-4}} \isdef \{(2,1,1),(1,2,1),(1,1,2)\}, \text{ and}\\ 
	&\mathemph{\oldvec{D}_{-6}} \isdef \{(2,2,2) \}.
	\end{align*}

\begin{lem}\label{Lem_TriangleInclusionAdjacency}
    Let $\mu: H \to F\subseteq G$ be a hexagonal chart of the 
    \mammal\enspace $G$. Let  $\vec{s},\vec{t}\in \Z^3$ and  
    $k\in \{0,2,4,6\}$ and $m\geq k$,
    be such that the images of $\Delta_m^{\vec{s}}$ and 
    $\Delta_{m-k}^{\vec{t}}$ are subsets of $H$.  Further, let $S \subset F$ be
    the image of $\mu \circ \Delta_m^{\vec{s}}$ and $T \subset F$ the
    image of $\mu \circ \Delta_{m-k}^{\vec{t}}$. 
    Then, $S$ and $T$ are adjacent
    in the clique graph $G_n$ for all $n \geq m$ with 
    $n \equiv_2 m$ if and only if $\vec{t}-\vec{s} \in \vec{D}_{-k}$.
\end{lem}
\begin{proof}
    Since $\mu$ is an isomorphism, $S$ and $T$ are adjacent in $G_n$ if 
    and only if  the images of 
    $\Delta_m^{\vec{s}}$ and $\Delta_{m-k}^{\vec{t}}$ are connected 
    by an edge of the $n$-th iterated geometric clique graph 
    $(\Hex_{m+|\vec{s}|})_n$ of the hexagonal grid. Therefore, it is sufficient 
    to prove the claim
    for $G=\Hex_m$. Since 
    \begin{equation*}
        \Hex_{m+|\vec{s}|} \to \Hex_m, \qquad a \mapsto a-\vec{s}
    \end{equation*}
    is an isomorphism between hexagonal grids, 
    we can assume without loss of generality, that
    $S = \Delta_m$ and $T$ is the image of $\Delta_{m-k}^{\vec{t}-\vec{s}}$ with corners %$\left(\begin{array}{c} m-k\\0\\0\end{array}\right)$
    $(m-k,0,0)+\vec{t}-\vec{s}$,  $(0,m-k,0)+\vec{t}-\vec{s}$, and $(0,0,m-k)+\vec{t}-\vec{s}$.
    Now, we distinguish with respect to $k$:
    \begin{enumerate}
        \item $k=0$: $T$ is adjacent to $S$ if the corners of $T$ lie
            in the neighbourhood $\neig{G}{S}$. A vertex $(v_1,v_2,v_3)\in \Hex_m$
            lies
            in $\neig{G}{\Delta_m}$ if and only if $-1 \leq v_i \leq m+1$.
            Since the components of $\vec{t}-\vec{s}$ sum to 0, this is
            equivalent to $\vec{t}-\vec{s}\in \vec{D}_0$.
        \item $k=2$: $T \subset S$ if and only if the corners of $T$ lie
            in $S$. Equivalently, 
            all components of $\vec{t}-\vec{s}$ have to be non-negative. Since the 
            components sum to 2, this is equivalent to $\vec{t}-\vec{s} \in \vec{D}_{-2}$.
        \item $k=4$: The corners of $T$ do not lie on the boundary if and 
            only if all components of $\vec{t}-\vec{s}$ are at least 1. Since
            the components sum to 4, this is equivalent to $\vec{t}-\vec{s} \in \vec{D}_{-4}$.
        \item $k=6$: The corners of $T$ have distance 2 from the boundary
            of $S$ if and only if all components of $\vec{t}-\vec{s}$ are at least 2.
            Since the components sum to 6, 
            this is equivalent to $\vec{t}-\vec{s} \in \vec{D}_{-6}$. \qedhere
    \end{enumerate}
\end{proof}

\noindent From every clique we can choose an element $S$ of maximal level $m$.
Then, we describe the clique as a clique of the lower-level neighbourhood 
$N_{G_n}[S]\cap V(G_m)$. To describe $N_{G_n}[S]$ combinatorially, we 
introduce the \emph{local hexagonal graph}: Its vertices are
    \begin{equation*}
        \mathemph{V_{LHG}}\isdef\left\{v_0^{0,0,0}\right\} \cup \left\{v_r^{\vec{d}} \mid r \in \{0,-2,-4,-6\}, \vec{d} \in \vec{D}_r \right\}
    \end{equation*}
    and its edges are given by
    \begin{equation*}
        \mathemph{E_{LHG}}\isdef\left\{(v_{r}^{\vec{x}},v_{r-k}^{\vec{y}})\mid \vec{y}-\vec{x} \in \vec{D}_{-k}\text{ for a }k\in\{0,2,4,6\} \right\}.
    \end{equation*}
    For a set $Q$ of vertices of $G_n$ of a given level $m$, the 
    \emph{lower level neighbourhood} of $Q$ is defined as the 
    set $\cneig{G_m}{Q} \subseteq G_n$, which consists of all the 
    common neighbours of the elements in $Q$ that have a level of at most $m$.

\begin{lem}\label{Lem_LowerLevelNeighbourhoodToLHG}
    Let $S \cong \Delta_m$ be a vertex in $G_n$ with $m \geq 3$. 
    The lower-level-neighbourhood of $S$ in $G_n$ is 
    isomorphic to an induced subgraph of 
    the local hexagonal
    graph.
\end{lem}
\begin{proof}
    We give a graph monomorphism $\phi\colon N_{G_n}[S]\cap G_m\to LHG$ 
    that maps non-edges to non-edges.
    We start with the generic case $m \geq 6$ and a standard chart 
    $\Delta_m \to S$. By Lemma \ref{Lem_HexagonalChartExtension}, 
    we can extend it to a hexagonal chart 
    $\mu: E \to G$ such that all adjacent $T \cong \Delta_m$ are 
    contained.
    We have the following adjacencies of smaller level:
    \begin{enumerate}
        \item The inclusions of $\Delta_{m-2}$ into $\Delta_m= S$ are all described
            by triangle inclusion maps since $m > 4$ 
            (Lemma \ref{Lem_TriangleInclusionTwo}).
        \item The inclusions of $\Delta_{m-4}$ into $\Delta_{m-3}$ (compare
            Remark \ref{Rem_TriangleInclusionBoundaryDistance}) are all
            described by triangle inclusion maps since $m-3 > 2$ 
            (Lemma \ref{Lem_TriangleInclusionOne}).
        \item The inclusion of $\Delta_{m-6}$ into $\Delta_{m-6}$ (compare
            Remark \ref{Rem_TriangleInclusionBoundaryDistance}) is unique
            and also given by a triangle inclusion map.
    \end{enumerate}
    Thus, all adjacent
    triangles of smaller level are given by triangle inclusion maps. 
    Therefore, by Lemma \ref{Lem_TriangleInclusionAdjacency}, 
    $\phi(\mu(\Delta_{m-k}^{\vec{e}}))=v_{-k}^{\vec{e}}$ is a monomorphism 
    of the required property, but it is not necessarily an isomorphism since not all 
    of the $\vec{D}_0$-translated neighbours of $S$ need to  be present. 
    
    We continue with the case $m=5$, illustrated in Figure \ref{Fig_AdjacencyFive}. 
    Since $5 > 4$, all neighbours of level
    $m-2$ are given by triangle inclusion maps (Lemma \ref{Lem_TriangleInclusionTwo}).
    For level $m-4=1$, we need to consider inclusions of $\Delta_1$ into
    $\Delta_{2}$ (Remark \ref{Rem_TriangleInclusionBoundaryDistance}). By
    Lemma \ref{Lem_TriangleInclusionOne}, one exceptional case occurs: a 
    graph $T \cong \Delta_1$ with vertices $(2,1,2)$, $(2,2,1)$, and $(1,2,2)$. 
    However, there is no neighbour of level $m-6$ since $m-6=-1$. Thus, 
    we define $\phi$ as in the generic case, but we map $T$ to $v_{-6}^{2,2,2}$.
    Since Lemma \ref{Lem_TriangleInclusionAdjacency} shows the correct edge 
    correspondence for all neighbours given by triangle inclusion maps, it remains
    to show that the edges of the local hexagonal graph correctly describe the
    adjacencies of $v_{-6}^{2,2,2}$.
    
    \begin{figure}[htbp]
    	\centering
    	\def\txtsize{\small}
    	\begin{tikzpicture} [scale=0.8, inner sep=0.1]
    	\HexagonalCoordinates{5}{5}
    	
    	\foreach \a/\b/\c/\col/\l in {11/21/12/yellow/{$v_{-4}^{2,1,1}$}, 21/31/22/yellow/{$v_{-4}^{1,2,1}$}, 12/22/13/yellow/{$v_{-4}^{1,1,2}$}, 21/22/12/lime/{$v_{-6}^{2,2,2}$}}{
    		\fill[\col] (A\a) -- (A\b) -- (A\c) -- cycle;
    		\node at (barycentric cs:A\a=1,A\b=1,A\c=1) {\txtsize\l};
    	}
    	
    	\foreach \i in {0,...,4}{
    		\pgfmathparse{int(5-\i)}
    		\pgfmathsetmacro{\opp}{\pgfmathresult}
    		\draw (A0\i) -- (A\opp\i);
    		\draw (A\i0) -- (A\i\opp);
    		\draw (A0\opp) -- (A\opp0);
    	}
    	\foreach \a/\b/\c in {14/11/41, 01/31/04, 10/13/40}{
    		\draw[red, line width=1 mm] (A\a) -- (A\b) -- (A\c) -- cycle;
    	}

    	\foreach \a/\b/\c/\d/\e in { 13/10/40/red/0.55, 31/01/04/red/0.55, 11/41/14/red/0.55}{
    		\foreach \x in {0,0.05,0.1,0.15,0.2,0.25,0.3,0.35,0.4,0.45,0.5,0.55,0.6,0.65,0.7,0.75,0.8,0.85,0.9,0.95,1}{
    			\draw[\d,line width=\e mm](barycentric cs:A\a=\x,A\b=1-\x)--(barycentric cs:A\a=\x,A\c=1-\x);
    		}
    	}
    	
    	\foreach \a/\b/\c in {02/32/05, 00/30/03, 20/23/50}{
    		\draw[blue, line width=1mm] (A\a) -- (A\b) -- (A\c) -- cycle;
    	}
    	
    	\foreach \a/\b/\c/\d/\e in { 05/02/32/blue/0.55, 50/20/23/blue/0.55, 00/30/03/blue/0.55}{
    		\foreach \x in {0,0.05,0.1,0.15,0.2,0.25,0.3,0.35,0.4,0.45,0.5,0.55,0.6,0.65,0.7,0.75,0.8,0.85,0.9,0.95,1}{
    			\draw[\d,line width=\e mm](barycentric cs:A\a=\x,A\b=1-\x)--(barycentric cs:A\a=\x,A\c=1-\x);
    		}
    	}
    	
    	\foreach \a/\b in {10/40, 01/04, 41/14}{
    		\draw[red, line width=0.3 mm] (A\a) -- (A\b);
    	}
    	
    	\node[circle,draw=red, fill=white] at (A22) {\tiny \textcolor{red}{$v_{-2}^{0,1,1}$}};
    	\node[circle,draw=red, fill=white] at (A21) {\tiny \textcolor{red}{$v_{-2}^{1,1,0}$}};
    	\node[circle,draw=red, fill=white] at (A12) {\tiny \textcolor{red}{$v_{-2}^{1,0,1}$}};
    	
    	%	\node[blue] at (barycentric cs:A22=1,A32=1,A23=1) {\txtsize $v_{-2}^{0,1,1}$};
    	%	\node[blue] at (barycentric cs:A20=1,A30=1,A21=1) {\txtsize $v_{-2}^{1,1,0}$};
    	%	\node[blue] at (barycentric cs:A02=1,A12=1,A03=1) {\txtsize $v_{-2}^{1,0,1}$};

    	%	\foreach \i in {0,...,4}{
    	%		\pgfmathparse{int(5-\i)}
    	%		\pgfmathsetmacro{\opp}{\pgfmathresult}
    	%		\draw (A0\i) -- (A\opp\i);
    	%		\draw (A\i0) -- (A\i\opp);
    	%		\draw (A0\opp) -- (A\opp0);
    	%}

    	\node[circle,draw=blue, fill=white] at (A13) {\tiny \textcolor{blue}{$v_{-2}^{0,0,2}$}};
    	\node[circle,draw=blue, fill=white] at (A31) {\tiny \textcolor{blue}{$v_{-2}^{0,2,0}$}};
    	\node[circle,draw=blue, fill=white] at (A11) {\tiny \textcolor{blue}{$v_{-2}^{2,0,0}$}};	
    	
    	%	\node[red] at (barycentric cs:A04=1,A14=1,A13=1) {\txtsize $v_{-2}^{0,0,2}$};
    	%	\node[red] at (barycentric cs:A11=1,A10=1,A01=1) {\txtsize $v_{-2}^{2,0,0}$};
    	%	\node[red] at (barycentric cs:A40=1,A41=1,A31=1) {\txtsize $v_{-2}^{0,2,0}$};
    	
    	\foreach \a/\b/\c/\col/\l in {11/21/12/yellow/{$v_{-4}^{2,1,1}$}, 21/31/22/yellow/{$v_{-4}^{1,2,1}$}, 12/22/13/yellow/{$v_{-4}^{1,1,2}$}, 21/22/12/lime/{$v_{-6}^{2,2,2}$}}{
    		%\fill[\col] (A\a) -- (A\b) -- (A\c) -- cycle;
    		\node[circle,fill=\col] at (barycentric cs:A\a=1,A\b=1,A\c=1) {\tiny\l};
    	}
    	\end{tikzpicture}
    	\caption{Adjacencies for $m=5$}
    	\label{Fig_AdjacencyFive}
    \end{figure}
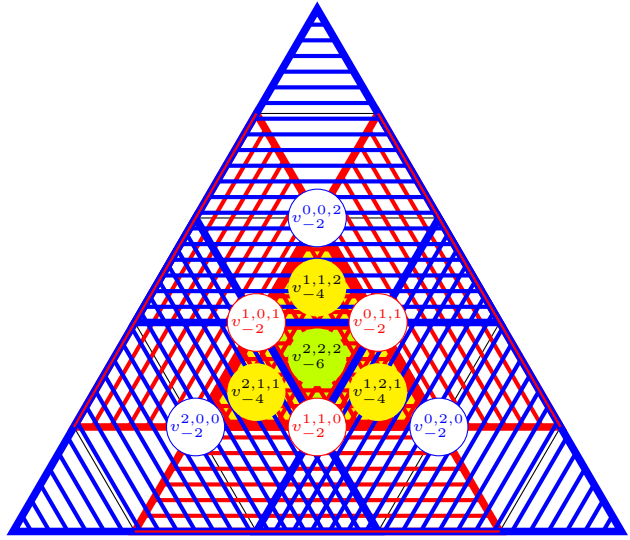

    \begin{itemize}
        \item By definition, the middle $\Delta_1$ is adjacent to $\Delta_5$ in $G_n$ as well as $v_0^{0,0,0}$ and $v_{-6}^{2,2,2}$ are adjacent in $LHG$.
        \item Furthermore, $T$ is adjacent to
            the three triangles $\Delta_3^{0,1,1}, \Delta_3^{1,0,1}$, and 
            $\Delta_3^{1,1,0}$ (as an exceptional case in
            Lemma 
            \ref{Lem_TriangleInclusionTwo}), exactly as in the
            local hexagonal graph.
        \item Since $T$ is adjacent to all three relevant 
            $\Delta_1$, the description of the 
            local hexagonal graph is correct again.
    \end{itemize}

    Next, we move on to $m=4$. Again, the only difference to the 
    generic case is the designated preimage of $v_{-6}^{2,2,2}$, 
    which is a $\nabla_2$ with corners
    $(2,2,0)$, $(0,2,2)$, and $(2,0,2)$. 
    As can be seen in Figure \ref{Fig_AdjacencyFour}, we check the adjacencies to the levels 0 and 2. Both are 
    satisfied again. 
 \begin{figure}[htbp]
	\centering
	\begin{tikzpicture} [scale=1.02,inner sep=1.0]
	\HexagonalCoordinates{5}{5}
	
	\foreach \a/\b/\c in {33/31/13}{
		\fill[yellow] (A\a) -- (A\b) -- (A\c) -- cycle;
	}
	
	\foreach \a/\b/\c in {11/31/13}{
		
		\foreach \x in {0.05,0.1,0.15,0.2,0.25,0.3,0.35,0.4,0.45,0.5,0.55,0.6,0.65,0.7,0.75,0.8,0.85,0.9,0.95,1}
		\draw[red,line width=0.4mm](barycentric cs:A\a=\x,A\b=1-\x)--(barycentric cs:A\a=\x,A\c=1-\x);
	}

	\foreach \a/\b/\c in {33/31/51}{
		
		\foreach \x in {0.05,0.1,0.15,0.2,0.25,0.3,0.35,0.4,0.45,0.5,0.55,0.6,0.65,0.7,0.75,0.8,0.85,0.9,0.95,1}
		\draw[red,line width=0.4mm](barycentric cs:A\c=\x,A\a=1-\x)--(barycentric cs:A\c=\x,A\b=1-\x);
	}

	\foreach \a/\b/\c in {33/13/15}{
		
		\foreach \x in {0.05,0.1,0.15,0.2,0.25,0.3,0.35,0.4,0.45,0.5,0.55,0.6,0.65,0.7,0.75,0.8,0.85,0.9,0.95,1}
		\draw[red,line width=0.4mm](barycentric cs:A\c=\x,A\a=1-\x)--(barycentric cs:A\c=\x,A\b=1-\x);
	}

	\foreach \a/\b/\c in {22/42/24}{
		
		\foreach \x in {0.05,0.1,0.15,0.2,0.25,0.3,0.35,0.4,0.45,0.5,0.55,0.6,0.65,0.7,0.75,0.8,0.85,0.9,0.95,1}
		\draw[blue,line width=0.4mm](barycentric cs:A\a=\x,A\b=1-\x)--(barycentric cs:A\a=\x,A\c=1-\x);
	}

	\foreach \a/\b/\c in {12/32/14}{
		
		\foreach \x in {0.05,0.1,0.15,0.2,0.25,0.3,0.35,0.4,0.45,0.5,0.55,0.6,0.65,0.7,0.75,0.8,0.85,0.9,0.95,1}
		\draw[blue,line width=0.4mm](barycentric cs:A\b=\x,A\a=1-\x)--(barycentric cs:A\b=\x,A\c=1-\x);
	}

	\foreach \a/\b/\c in {21/23/41}{
		
		\foreach \x in {0.05,0.1,0.15,0.2,0.25,0.3,0.35,0.4,0.45,0.5,0.55,0.6,0.65,0.7,0.75,0.8,0.85,0.9,0.95,1}
		\draw[blue,line width=0.4mm](barycentric cs:A\b=\x,A\a=1-\x)--(barycentric cs:A\b=\x,A\c=1-\x);
	}    
	
	\draw (A11) -- (A51) -- (A15) -- (A11)
	(A12) -- (A42) -- (A41) -- (A14) -- (A24) -- (A21) -- (A12)
	(A13) -- (A31) -- (A33) -- (A13);
	
	%            \draw[very thick, red] (A20) -- (A00) -- (A02) -- cycle;
	%            \draw[very thick, red] (A20) -- (A40) -- (A22) -- cycle;
	%            \draw[very thick, red] (A02) -- (A04) -- (A22) -- cycle;
	%            \draw[very thick, blue] (A21) -- (A01) -- (A03) -- cycle;
	%            \draw[very thick, blue] (A31) -- (A11) -- (A13) -- cycle;
	%            \draw[very thick, blue] (A10) -- (A30) -- (A12) -- cycle;
	%            \draw[very thick, red, dashed] (A00) -- (A40) -- (A04) -- cycle;

	\foreach \p in {22, 32, 23}{
		\fill[green!80!black] (A\p)  circle (3pt);
		
		\foreach \a/\b/\c/\e in {22/12/21/{2,0,0},31/51/33/{0,2,0}, 13/33/15/{0,0,2}}{
			\node[circle, draw=red,fill=white] at (barycentric cs:A\a=1,A\b=1,A\c=1) (l\a\b\c) {\small \textcolor{red}{$v_{-2}^{\e}$}};
			
		}	
	} 
	
	\foreach \a/\b/\c/\e in {21/41/23/{1,1,0}, 12/32/14/{1,0,1}, 22/42/24/{0,1,1}}{
		\node[circle, draw=blue,fill=white] at (barycentric cs:A\a=1,A\b=1,A\c=1) {\small\textcolor{blue}{$v_{-2}^{\e}$}};
	}	
	
	\foreach \a/\b/\c/\e in {40/31/30/{2,1,1}, 43/34/33/{1,2,1}, 03/04/13/{1,1,2}}{
		\node[circle, draw=black,fill=green!80!black] at (barycentric cs:A\a=1,A\b=1,A\c=1) (l\a) {\tiny$v_{-4}^{\e}$};
	}
	\draw (l40) edge[black,out=160,in=240,->] (A22); 
	\draw (l43) edge[black,out=270,in=0,->] (A32);
	\draw (l03) edge[black,out=40,in=120,->] (A23);
	
	\node[circle,draw=black,fill=yellow] at (U22) {\tiny\textcolor{black}{$v_{-6}^{2,2,2}$}};
	%            \node[anchor=west,shape=circle,fill=yellow] at (barycentric cs:A33=1,A43=1,A34=1) (label) {\textcolor{black}{$v_{-6}^{2,2,2}$}};
	%            \node[anchor=east] at (U22) (m){};
	%            \draw (label) edge[black,out=180,in=0,->] (m); 
	
	\end{tikzpicture}
	\caption{Adjacencies for $m=4$ }
	\label{Fig_AdjacencyFour}
\end{figure}
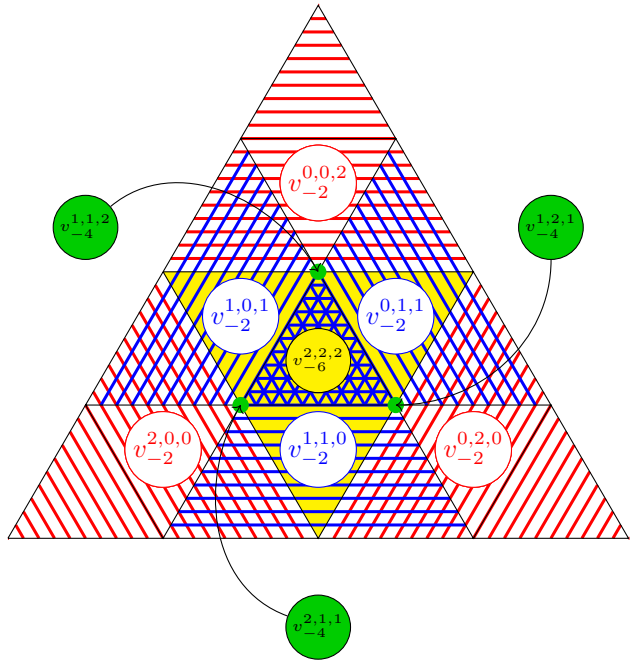

    Finally, we deal with $m=3$, with several differences to the generic case:
    \begin{enumerate}
        \item There is another adjacent $\Delta_3$ adjacent to $S$ which is ``facing down''.
            We denote it by $T$ and set
            $\phi(T)=v_{-6}^{2,2,2}$.
        \item There are the three subgraphs isomorphic to $\Delta_1$ from Lemma
            \ref{Lem_TriangleInclusionTwo} adjacent to $S$, called
            $T_{1,0,0}, T_{0,1,0}$, and $T_{0,0,1}$, and we map 
            them by $\phi(T_{1,0,0})=v_{-4}^{2,1,1}$, 
            $\phi(T_{0,1,0})=v_{-4}^{1,2,1}$, and 
            $\phi(T_{0,0,1})=v_{-4}^{1,1,2}$. 
    \end{enumerate}
    Figure \ref{Fig_AdjacencyThree}  shows that
    the local hexagonal graph 
    describes the adjacencies correctly.\qedhere

     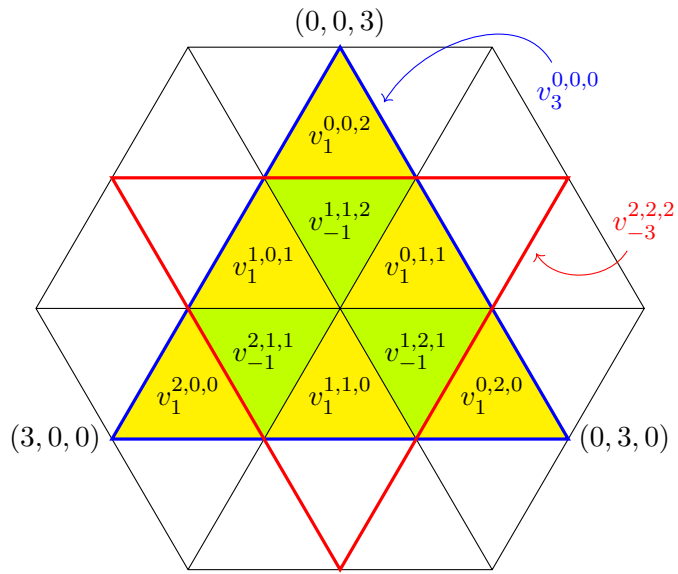
\begin{figure}[htbp]
     	\centering
        \begin{tikzpicture} %[scale=0.75]
        
            \HexagonalCoordinates{4}{4}
    
            \foreach \a/\b/\c/\e in {11/21/12/{2,0,0}, 21/31/22/{1,1,0}, 31/41/32/{0,2,0}, 12/22/13/{1,0,1}, 22/32/23/{0,1,1}, 13/23/14/{0,0,2}}{
                \fill[yellow] (A\a) -- (A\b) -- (A\c) -- cycle;
                \node at (barycentric cs:A\a=1,A\b=1,A\c=1) {$v_1^{\e}$};
            }
            \foreach \a/\b/\c/\e in {21/12/22/{2,1,1}, 22/31/32/{1,2,1}, 13/23/22/{1,1,2}}{
                \fill[lime] (A\a) -- (A\b) -- (A\c) -- cycle;
                \node at (barycentric cs:A\a=1,A\b=1,A\c=1) {$v_{-1}^{\e}$};
            }
    
            \draw (A11) -- (A41) -- (A14) -- (A11)
            (A30) -- (A33) -- (A03) -- (A30)
            (A20) -- (A40) -- (A42) -- (A24) -- (A04) -- (A02) -- (A20)
            (A20) -- (A24)
            (A40) -- (A04)
            (A02) -- (A42);
            
            \draw[very thick, blue] (A11) -- (A41) -- (A14) -- cycle;
            \node[blue] at (barycentric cs:A24=1,A33=1,A34=1) (l1) {$v_3^{0,0,0}$};
            \node at (barycentric cs:A14=1,A23=1) (m1) {};
             \draw (l1) edge[blue,out=120,in=60,->] (m1);

            \draw[very thick, red] (A03) -- (A30) -- (A33) -- cycle;
            \node[red] at (barycentric cs:A42=1,A33=1,A43=1) (l2) {$v_{-3}^{2,2,2}$};
            \node at (barycentric cs:A33=1,A32=1) (m2) {};
            \draw (l2) edge[red,out=240,in=300,->] (m2);
            
            \foreach \p/\r/\l in {11/left/{$(3,0,0)$}, 41/right/{$(0,3,0)$}, 14/above/{$(0,0,3)$}}{
                \node[\r] at (A\p) {\l};
            }
        \end{tikzpicture}
        \caption{Adjacencies for $m=3$}
        \label{Fig_AdjacencyThree}
    \end{figure}
\end{proof}

\noindent We describe the cliques of the local hexagonal graph.

\begin{lem}\label{Lem_CliquesOfLocalHexagonalGraph}
    Let $C$ be a clique in the local hexagonal graph with $v_0^{0,0,0} \in C$.
    Then, one of the following three cases holds:
    \begin{align*}
        1.\quad C = \mathemph{C_{-6}} &\isdef \{v_0^{0,0,0},v_{-2}^{1,1,0}, v_{-2}^{0,1,1}, 
            v_{-2}^{1,0,1}, v_{-4}^{2,1,1}, v_{-4}^{1,2,1}, v_{-4}^{1,1,2}, 
            v_{-6}^{2,2,2}\}\\
            &\ =\cneig{LHG}{v_{-4}^{2,1,1}, v_{-4}^{1,2,1}, v_{-4}^{1,1,2}}=\neig{LHG}{v_{-6}^{2,2,2}},\\
        2.\quad C = \mathemph{C_{-4}^{\oldvec{e}}} &\isdef 
            \{ v_0^{-(1,0,0)+\vec{e}}, v_0^{-(0,1,0)+\vec{e}}, v_0^{-(0,0,1)+\vec{e}},\\
            &\qquad v_{-2}^{(1,0,0)+\vec{e}}, v_{-2}^{(0,1,0)+\vec{e}}, v_{-2}^{(0,0,1)+\vec{e}}, v_{-4}^{(1,1,1)+\vec{e}}\} \\
            &\ =\cneig{LHG}{v_{-2}^{(1,0,0)+\vec{e}}, v_{-2}^{(0,1,0)+\vec{e}}, v_{-2}^{(0,0,1)+\vec{e}}}\\
            &\ =\neig{LHG}{v_{-2}^{2\vec{e}}} \text{ for an $\vec{e} \in \vec{E}$,}\\
        3. \quad C = \mathemph{C_{-2}^{\oldvec{e}}} &\isdef
            \{ v_0^{(1,0,0)-\vec{e}}, v_0^{(0,1,0)-\vec{e}}, 
            v_0^{(0,0,1)-\vec{e}}, v_{-2}^{(1,1,1)-\vec{e}} \}\\
            &\ =\cneig{LHG}{v_0^{(1,0,0)-\vec{e}}, v_0^{(0,1,0)-\vec{e}}, v_0^{(0,0,1)-\vec{e}}} \text{ for an $\vec{e} \in \vec{E}$.}
    \end{align*}
\end{lem}
\begin{proof}
    By the definition of the local hexagonal graph, the given sets form 
    complete subgraphs. Furthermore, they are represented as common 
    neighbourhoods of triangles or as closed neighbourhoods of vertices in the claimed way. Thus, they are also maximal.
    It remains to show that there cannot be any other cliques.
    
    If $v_{-6}^{2,2,2} \in C$, we get $C=\neig{LHG}{v_{-6}^{2,2,2}}$ since this neighbourhood already  forms a clique. Thus, the first case of the lemma holds.
    
    Otherwise, $C$ contains an element not incident to $v_{-6}^{2,2,2}$. Those elements are either given by $v_{-2}^{2\vec{e}}$ for an $\vec{e}\in\vec{E}$ or by $v_{0}^{\vec{e_2}-\vec{e_1}}$ for $\vec{e_1},\vec{e_2}\in\vec{E}$ with $\vec{e_1}\neq\vec{e_2}$.  
    %Since there is a canonical action of $S_3$ on the local hexagonal graph,
    %it is sufficient to consider the elements $v_{-2}^{0,2,0}$ and
    %$v_{0}^{-1,1,0}$.
    
    If $v_{-2}^{2\vec{e}} \in C$, we get $C=\neig{LHG}{v_{-2}^{2\vec{e}}}$ since this neighbourhood already forms a clique. Thus, the second case of the lemma holds.  This neighbourhood is a clique.

    Finally, we assume $v_0^{\vec{e_2}-\vec{e_1}}\in C$, but $v_{-2}^{2\vec{e_2}}\notin C$ 
    (the other two vertices $v_{-2}^{2\vec{e}}$ with $\vec{e}\in\vec{E}$ are not adjacent to $v_0^{\vec{e_1}-\vec{e_2}}\in C$, anyway).
    For reasons of symmetry, we can choose $e_1=(1,0,0)$ and $e_2=(0,1,0)$. Thus, we have  $v_0^{-1,1,0} \in C$, but $v_{-2}^{0,2,0} \not\in C$.
    The set of neighbours of $v_0^{-1,1,0}$ is
    \begin{equation*}
        \left\{ v_0^{0,0,0}, v_0^{0,1,-1}, v_0^{-1,0,1}, v_{-2}^{1,1,0}, v_{-2}^{0,2,0}, v_{-4}^{1,2,1} \right\}.
    \end{equation*}
    Only $v_0^{-1,0,1}$ is not adjacent to $v_{-2}^{0,2,0}$, so it
    has to lie in $C$ since $C$ is maximal.
    Then, $C=\cneig{LHG}{v_0^{-1,1,0},v_0^{-1,0,1},v_0^{0,0,0}}=C_{-2}^{e_1}$, as
    described in the third case of the lemma.
\end{proof}

\noindent The following lemma describes how we can find the cliques of an induced 
subgraph using the cliques of the surrounding graph.

\begin{lem}\label{Lem_cliquesofinducedsubgraphs}
    For a graph $G$ and an induced subgraph 
    $H$,  every clique of $H$ is given as the intersections of a 
    (not necessarily unique) clique of $G$ with $H$.
\end{lem}	
\begin{proof}
    Let $C$ be a clique of $H$. Then, $C$ is a complete subgraph of $G$. 
    Therefore, there is at least one  clique $C_G$ of $G$ containing $C$. 
    Obviously, $C\subseteq C_G\cap H$. If there was an 
    $x\in C_G\cap H\setminus C$, the union $C\cup \{x\}$ were a complete subgraph 
    of $H$ since $H$ is an induced subgraph in contradiction to $C$ being chosen maximal.\qedhere
\end{proof}

We apply this to the image of a lower-level-neighbourhood under the 
embedding given in \ref{Lem_LowerLevelNeighbourhoodToLHG}. This way, we 
can classify all the cliques of $G_n$.

\begin{theo}\label{Thm_AllTheCliquesLarge}
    If $C$ is a clique of $G_n$ containing a vertex $\Delta_m$ with 
    $m\geq 3$, $C$ is given by the construction in 
    \ref{Lem_CliqueConstructionTriangle} or \ref{Lem_CliqueConstructionVertex}.
\end{theo}
\begin{proof}
    Let $C$ be a clique of $G_n$ and let $S\cong\Delta_m$ be a vertex with 
    maximal $m\geq 3$ of $C$. Thus, $C$ is contained in the lower level 
    neighbourhood of $S$. 
    The lower level neighbourhood of $S$ is isomorphic to an induced 
    subgraph $H$ of the local hexagonal graph containing $v_0^{0,0,0}$ and the 
    isomorphism $\mu$ maps $S$ to $v_0^{0,0,0}$.
     Thus, by Lemma \ref{Lem_cliquesofinducedsubgraphs}, 
     $C$ is isomorphic to the 
    intersection of $H$ with a clique $C_{LHG}$ of the local hexagonal graph.
    
    Thus, $C_{LHG}$ is one of the cliques given in Lemma 
    \ref{Lem_CliqueConstructionTriangle}. 
    For reasons of symmetry, we can restrict our investigation to the cliques 
    $C_{-6}, C_{-4}^{1,0,0},$ and $C_{-2}^{1,0,0}$. 
    \begin{enumerate}
        \item If $C_{LHG}=C_{-6}$ and $m\geq 4$, the preimages of 
            $v_{-4}^{2,1,1}, v_{-4}^{1,2,1},$ and $v_{-4}^{1,1,2}$ are 
            subgraphs of $S$ isomorphic to $\Delta_{m-4}$. Thus, they do 
            exist and $C$ is given by the construction of Lemma 
            \ref{Lem_CliqueConstructionTriangle}.
            If $m=3$, the preimages of  $v_{-4}^{2,1,1}, v_{-4}^{1,2,1},$ and 
            $v_{-4}^{1,1,2}$ do exist, but they are not contained in a common 
            $\Delta_2$ and we cannot apply Lemma \ref{Lem_CliqueConstructionTriangle}.
            Therefore, we look at the preimages of 
            $v_{-2}^{1,1,0}, v_{-2}^{0,1,1}, v_{-2}^{1,0,1}, v_{-4}^{2,1,1}, v_{-4}^{1,2,1},$ 
            and $v_{-4}^{1,1,2}$ which do exist, since they are induced 
            subgraphs of $S$ isomorphic to $\Delta_1$. Furthermore, those preimages
            are the subgraphs isomorphic to $\Delta_1$ containing the middle vertex of $S$. 
            Thus $C$ is constructed from this vertex by Lemma 
            \ref{Lem_CliqueConstructionVertex}.
        \item If $C_{LHG}=C_{-4}^{1,0,0}$, the preimages of 
            $v_{-2}^{(1,0,0)+\vec{e}}, v_{-2}^{(0,1,0)+\vec{e}}$ and 
            $v_{-2}^{(0,0,1)+\vec{e}}$ are subgraphs of $S$ isomorphic 
            to $\Delta_{m-2}$. Thus, they exist and $C$ is given 
            by the construction of Lemma \ref{Lem_CliqueConstructionTriangle}.
        \item If $C_{LHG}=C_{-2}^{1,0,0}$, either the preimages of 
            $v_0^{(1,0,0)-\vec{e}}, v_0^{(0,1,0)-\vec{e}}, v_0^{(0,0,1)-\vec{e}}$ exist 
            and $C$ is their common neighbourhood, or one of them does 
            not exist. In the second case, without loss of generality, 
            $\vec{e}=(1,0,0)$. Thus $v_0^{(1,0,0)-\vec{e}}=v_{0,0,0}$ and there is no
            preimage of $v_0^{(0,0,1)-\vec{e}}=v_0^{1,0,-1}$. The remaining elements 
            of $C_2^{\vec{e}}$ are at most $v_{0,0,0}, v_{-1,1,0}$ and 
            $v_{-1}^{0,1,1}$ which also lie in $C_{-4}^{0,1,0}$. Hence, we can 
            also see $C$ as the intersection of $LHG$ with $C_{-4}^{0,1,0}$ and, by 
            applying the second case, it is given by the construction of 
            Lemma \ref{Lem_CliqueConstructionTriangle}.\qedhere
    \end{enumerate}
\end{proof}

\noindent We finally managed to prove the surjectivity of the map from Remark \ref{lem_map} between the vertices of $G_{n+1}$ and the cliques of $G_{n}$. We continue proving the injectivity.

\begin{theo}\label{Theo_bijective}
    The map 
    \begin{align*}
        C\colon V(G_{n+1})&\to \{\text{cliques of }G_n\},\\
        S&\mapsto\text{the clique from}
        \begin{cases}
            \text{Lemma \ref{Lem_CliqueConstructionVertex}}, &\text{if $S$ is of level $0$},\\
            \text{Lemma \ref{Lem_CliqueConstructionTriangle}}, &\text{ otherwise,} 
        \end{cases}
    \end{align*}
    is bijective.
\end{theo}
\begin{proof}
    The map $C$ is surjective by Lemma \ref{Lem_EvenCliquesOfSmallLevel}, Lemma \ref{Lem_OddCliquesofSmallLevel}, and Theorem \ref{Thm_AllTheCliquesLarge}.
    For the injectivity we discuss three cases.
    The cliques from Lemma \ref{Lem_CliqueConstructionVertex} contain at 
    least six $\Delta_1$-shaped subgraphs of $G$ and  the cliques from 
    Lemma \ref{Lem_CliqueConstructionTriangle} contain at most four of them. 
    Thus, a clique, which is constructed through Lemma 
    \ref{Lem_CliqueConstructionTriangle}, cannot be constructed through 
    Lemma \ref{Lem_CliqueConstructionVertex} and vice versa.
    
    Furthermore, for an $S\cong\Delta_0$, the umbrella in $C(S)$ (recall
    Subsection \ref{Subsect_CliqueConstruction}) has a 
    unique central vertex, which is the vertex of $S$, and the preimage 
    of $C(S)$ is unique. Finally, let $S\cong \Delta_m$ for an 
    $m\geq 1$. We look for a $T\cong\Delta_{m'}$ such that $C(S)=C(T)$. 
    To do this, we check the options for $T_1,T_2,T_3\cong\Delta_{m'-1}$ 
    like in Lemma \ref{Lem_CliqueConstructionTriangle}. By Corollary 
    \ref{Cor_CliqueSummary}, if $m\geq 3$, the only other three triangles 
    in $C(S)$ of a common level are the elements of level $m+1$, if all 
    three of them exist. But their union is not isomorphic to a 
    $\Delta_{m+2}$. For $m=2$, there is an additional element of level 
    $1$, but it does not form a $\Delta_2$ with two of the other three. 
    For $m=1$, there is an additional element of level $2$ , but it 
    does not form a $\Delta_3$ with two of the other three.
\end{proof}

 \noindent The following corollary gives an explicit description of the cliques.

\begin{cor}\label{Cor_CliqueSummary}  
    \begin{itemize}
        \item[1.] For $m\geq 1$ and $\Delta_{m}\iso{\mu}S\in V(G_{n+1})$,
           an explicit description of
            $C(S)$ is given through

            \resizebox{.9\hsize}{!}{$
                C(S)=
                \underbrace{M_{m-1}}_{\lvert \cdot \rvert=3}\cup 
                \underset{\lvert\cdot\rvert=0,\text{ if }n=m,}{\underbrace{M_{m+1}}_{\lvert\cdot \rvert\leq 3\phantom{,\text{ if }n=m,}}}
                \cup\underset{\lvert\cdot\rvert=0, \text{ if } n\leq m+2,}{\underbrace{M_{m+3}}_{\lvert\cdot\rvert\leq 1\phantom{, \text{ if } n\leq m+2,}}} \cup
                \begin{cases}
                    \emptyset, &\text{ if }m=1\text{ and }n\leq 1,\\
                    \{\hat\mu(\Delta_4^{-(1,1,1)}(\nabla_{2}))\}, &\text{ if }m=1\text{ and }n\geq 2,\\
                    \{\mu(\nabla_{1})\},  &\text{ if }m=2,\\
                    \{ S\setminus \partial S\}, &\text{ if }m\geq 3.
                \end{cases} $}

            \begin{itemize}
                \item $M_{m-1}$ consists of the elements 
                    $\Delta_{m-1}\cong\mu^{\vec{e}}$ for $\vec{e}\in \vec{E}$.
                \item $M_{m+1}$ consists of the elements $\Delta_{m+1}\iso{\nu}T$ 
                    fulfilling $\mu=\nu^{\vec{e}}$ for an  $\vec{e}\in \vec{E}$.
                \item $M_{m+3}$ consists of the element $\Delta_{m+3}\cong T$ 
                    enclosing $S$ with distance $1$, i.\,e. $S=T\setminus \partial T$.
            \end{itemize}
            \item[2.] For $\Delta_0\cong S\in V(G_{n+1})$, if we denote the vertex of $S$ by $v$, an explicit description of
            $C(S)$ is given through 

                    \resizebox{.9\hsize}{!}{$
                    C(S)=
                    \underbrace{\{T\in V(G_n)\mid T\cong \Delta_1,\ S\subseteq T\}}_{\lvert\cdot\rvert=\deg_G(v)}
                    \cup 
                    \underset{\lvert \cdot \rvert=2,\text{ if } \deg_G(v)=6 \text{ and } n\geq 3,}{\underbrace{\{T\in V(G_n)\mid T\cong \Delta_3,\ S\subseteq T\setminus \partial T\}.}_{\lvert \cdot\rvert=0,\text{ if } \deg_G(v)\geq 7 \text{ or } n\leq 2,}\phantom{a}}$}

    \end{itemize}
\end{cor}

\section{Clique intersections of $\mathemph{G_n}$}\label{Sect_CliqueIntersections}
After having constructed all cliques of the geometric clique graph 
$G_n$ (and proven that these
cliques correspond to vertices of $G_{n+1}$), we need to show that
two cliques $C(S_1)$ and $C(S_2)$ intersect if and only if the 
corresponding vertices $S_1$ and $S_2$
in $G_{n+1}$ are connected by an edge. From now on, we assume 
that $S_1\cong \Delta_m$ and $S_2\cong \Delta_{m+k}$ for  $k\in\{0,2,4,6\}$ and $m\geq 0$.

\subsection{Case: $\mathemph{S_2 \cong \Delta_m}$}

\begin{lem}\label{Lem_NoLargeOrSmallTriangles}
    For $S_1,S_2\in V(G_{n+1})$ with $S_1 \cong S_2 \cong \Delta_m$ for some $m\geq 0$,
    the cliques $C(S_1)$ and $C(S_2)$ do not intersect in 
    a vertex $T\cong\Delta_{m+3}$. Furthermore, if $m\geq 4$, 
    they do not intersect in an element $T\cong \Delta_{m-3}$.
\end{lem}
\begin{proof}
    From Corollary \ref{Cor_CliqueSummary}, we see that if a clique 
    $C(S)$ contains an element $T\cong\Delta_{m+3}$, the clique is 
    uniquely defined by this element since $S=T\setminus \partial T$.
    Furthermore, for $m\geq 4$ the clique is also uniquely defined by 
    an element $T\cong \Delta_{m-3}$ since then $T=S\setminus\partial S$, which has only one solution $S$.
    In either way, the vertex $T$ cannot lie in two distinct cliques of $G_n$.
\end{proof}

\begin{lem}
    For $S_1,S_2\in V(G_{n+1})$ with $S_1\cong S_2\cong \Delta_{0}$, 
    the cliques $C(S_1)$ and $C(S_2)$ intersect non-trivially if and 
    only if $S_1$ and $S_2$ are adjacent in $G_{n+1}$, i.\,e.\ they 
    are adjacent in $G$.
\end{lem}
\begin{proof}
    At first, we suppose that $S_1$ and $S_2$ are adjacent in $G$. Since $G$ is locally cyclic, they have two common neighbours, there is a 
    $\Delta_1 \cong T \subset G$ with $S_1 \subset T$
    and $S_2 \subset T$. Thus, $T$ lies in both $C(S_1)$ and $C(S_2)$. 
    Conversely, suppose there is a $T\in C(S_1)\cap C(S_2)$. By 
    Lemma \ref{Lem_NoLargeOrSmallTriangles}, $T\cong \Delta_{1}$. 
    Furthermore, 
    by Corollary \ref{Cor_CliqueSummary}, $S_1$ and $S_2$ are both 
    vertices of $T$ and thus adjacent.
\end{proof}

\begin{lem}\label{Lem Fromm+1tom-1} 
    For $S_1,S_2\in V(G_{n+1})$ with $S_1 \cong S_2 \cong \Delta_m$ for some $m \geq 1$,
    if the cliques $C(S_1)$ and $C(S_2)$ intersect 
    in a $T_1\cong\Delta_{m+1}$, they also intersect in a 
    $T_2\cong\Delta_{m-1}$.
\end{lem}
\begin{proof}
    If $T\in C(S_1)\cap C(S_2)$ for a $\Delta_{m+1}\iso{\nu} T$, 
    both $S_1\subseteq T$ and $S_2\subseteq T$. If $m+1\geq 3$, it follows that $S_1=\nu^{\vec{e}}$ and $S_2=\nu^{\vec{f}}$ for 
    $\vec{e},\vec{f}\in \vec{E}$ and $\vec{e}\neq\vec{f}$, implying 
    $S_1 \cap S_2 \cong \Delta_{m-1}$.
    
    If $m+1=2$, either the situation is the same as in the foregoing 
    case or, without loss of generality, $S_1=\mu^{\vec{e}}$ for 
    an $\vec{e}\in \vec{E}$ and $S_2=\mu(\nabla_{1})$.  
    Even in this case, $S_1$ and $S_2$  intersect in a vertex.
\end{proof}

\noindent Consequently, if $m \geq 1$, we only have to investigate whether two cliques intersect in a  $\Delta_{m-1}$-shaped vertex of $G_n$. 

\begin{lem}
    For $S_1,S_2\in V(G_{n+1})$ with $S_1\cong S_2\cong \Delta_{m}$ 
    for an $m\geq 1$, the cliques $C(S_1)$ and $C(S_2)$ intersect 
    non-trivially if and only if $S_1$ and $S_2$ are adjacent in 
    $G_{n+1}$, i.\,e. $S_1\subseteq \neig{G}{S_2}$.
\end{lem}
\begin{proof}
    If there is a $T\in C(S_1)\cap C(S_2)$, by Lemma 
    \ref{Lem_NoLargeOrSmallTriangles} and Lemma \ref{Lem Fromm+1tom-1}, 
    we can choose $T \cong \Delta_{m-1}$. Thus, we have 
    $S_1 \subset\neig{G}{T} \subset \neig{G}{S_2}$, 
    where $S_1 \subset \neig{G}{T}$ follows from Lemma \ref{Lem_TriangleInclusionOne}.
	
    \noindent Conversely, suppose $S_1 \subset \neig{G}{S_2}$. We distinguish between the values of $m$.
	
    \begin{itemize}
	\item If $m=1$, $S_1$ is one of the additional faces
	    in $\neig{G}{S_2}$. Thus $S_1$ and $S_2$ intersect in at least one 
            vertex, which lies in both $C(S_1)$ and $C(S_2)$.
	\item If $m=2$, there is a $\Delta_1\cong T\subseteq S_1 \cap S_2$. 
            Thus, $T\in C(S_1)\cap C(S_2)$.
	\item If $m\geq 4$, let $\mu\colon \Delta_{m}\to S_1$ be a standard 
            chart. By Lemma \ref{Lem_HexagonalChartExtension}, there is an 
            extension $\hat{\mu}\colon E\to \hat{S}$ such that $S_2$ is the image 
            of $\hat{\mu}\circ \Delta_m^{\vec{t}}$ for a $\vec{t}\in \vec{D}_0$. 
            Therefore, 
            $S_1\cap S_2\cong \hat{\mu}^{-1}(S_1\cap S_2)=\Delta_m\cap \Delta_m^{\vec{t}}(\Delta_m)\cong \Delta_{m-1}$. 
            Thus, by definition of the cliques, 
            $S_1\cap S_2\in C(S_1)\cap C(S_2)$, so they intersect non-trivially. 
	\item If $m=3$, by Lemma \ref{Lem_HexagonalChartExtension}, we are either in the same situation as for $m\geq 4$, which proves the claim, 
            or $S_2$ lies twisted in the middle of $\neig{G}{S_1}$. In this case, 
            $C(S_1)$ and $C(S_2)$ share the vertex equivalent to the midpoint 
            of both $S_1$ and $S_2$. \qedhere
    \end{itemize}
\end{proof}
\noindent Thus, for $S_1\cong S_2\in V(G_{n+1})$, intersection in $G_n$ and adjacency in $G_{n+1}$ are equivalent.

\subsection{Case: $\mathemph{S_2\cong\Delta_{m+k}}$ for $\mathemph{k\in \{2,4,6\}}$}

\begin{lem}
    For $S_1,S_2\in V(G_{n+1})$ with $S_1\cong \Delta_{m}$ and 
    $S_2\cong \Delta_{m+2}$ for an $m\geq 0$, the cliques $C(S_1)$ 
    and $C(S_2)$ intersect non-trivially if and only if $S_1$ 
    and $S_2$ are adjacent in $G_{n+1}$, i.\,e. $S_1\subseteq S_2$.
\end{lem}
\begin{proof}
    At first, we suppose that $S_1\subseteq S_2$.
    By Lemma \ref{Lem_TriangleInclusionTwo}, if $m\neq 2$ and by choosing a chart $\Delta_{m+2}\iso{\mu} S_2$, we get $S_1=\mu^{\vec{e}+\vec{f}}$ for some  $\vec{e},\vec{f}\in \vec{E}$ and
    $T\isdef\mu^{\vec{e}}$  fulfils $S_1\subseteq T\subseteq S_2$. By 
    Corollary \ref{Cor_CliqueSummary}, $T$ lies in both cliques $C(S_1)$ and $C(S_2)$.

    If $m=2$, the triangle $S_1$ either lies inside a 
    $T\subseteq S_2$ which is isomorphic to $\Delta_3$ or $S_1$ 
    contains the unique $S\cong \Delta_1$, which has distance 
    $1$ to $\partial S_2$. In both cases, $S$ or $T$ respectively 
    lies in both $C(S_1)$ and $C(S_2)$.
	
    Conversely, we now suppose that $C(S_1)$ and $C(S_2)$ intersect 
    non-trivially.
    We distinguish between the possibilities for the element in the 
    intersection. Any element $T\in C(S_1)\cap C(S_2)$ is isomorphic 
    to $\Delta_{m-1},\ \Delta_{m+1},$ or $\Delta_{m+3}$. 
    \begin{itemize}
	\item If $T\cong \Delta_{m-1}$ (i.\,e. $m\geq 1$), $T$ has 
            distance $1$ to the boundary of $S_2$; thus 
            $T=S_2\setminus\partial S_2$. All the graphs isomorphic 
            to $\Delta_{m}$, which contain $T$, are subgraphs of 
            $S_2$; thus $S_1\subseteq S_2$.
	\item If $T\cong \Delta_{m+1}$, by Corollary 
            \ref{Cor_CliqueSummary}, this means 
            $S_1\subseteq T\subseteq S_2$, which proves the claim. 
	\item If $T\cong \Delta_{m+3}$, $S_1$ is the unique subgraph 
            of $T$ with distance $1$ from the boundary, i.\,e. 
            $S_1=T\setminus\partial T$. This subgraph is contained in every subgraph 
            of $T$ isomorphic to $\Delta_{m+2}$; thus $S_1\subseteq S_2$. \qedhere
    \end{itemize}	
\end{proof}	

\begin{lem}
    For $S_1,S_2\in V(G_{n+1})$ with $S_1\cong \Delta_{m}$ and 
    $S_2\cong \Delta_{m+4}$ for an $m\geq 0$, the cliques $C(S_1)$ and 
    $C(S_2)$ intersect non-trivially if and only if $S_1$ and $S_2$ 
    are adjacent in $G_{n+1}$, i.\,e. $S_1\subseteq S_2\setminus\partial S_2$.
\end{lem}
\begin{proof}
    If $S_1\subseteq S_2\setminus\partial S_2$, the element $S_2\setminus\partial S_2\cong \Delta_{m+1}$ lies in both cliques.
   
    \noindent Conversely, if the cliques intersect, they intersect in a 
    $T$ fulfilling $S_1\subseteq T \subseteq S_2$ with 
    $T \cong\Delta_{m+1}$. Further, the distance of $T$ and $\partial S_2$ is $1$
    or $T\cong\Delta_{m+3}$ and the distance of $S_1$ and $\partial T$ is $1$.
    Thus, the distance between $S_1$ and $\partial S_2$ is 1 and  $S_1\subseteq S_2\setminus\partial S_2$.
\end{proof}	

\begin{lem}
    For $S_1,S_2\in V(G_{n+1})$ with $S_1\cong \Delta_{m}$ and 
    $S_2\cong \Delta_{m+6}$ for an $m\geq 0$, the cliques $C(S_1)$ and 
    $C(S_2)$ intersect non-trivially if and only if $S_1$ and 
    $S_2$ are adjacent in $G_{n+1}$, i.\,e. $S_1\subseteq S_2\setminus\neig{G}{\partial S_2}$.
\end{lem}
\begin{proof}
     If 
     $S_1\subseteq S_2\setminus\neig{G}{\partial S_2}$, the only 
  $T \cong \Delta_{m+3}$ such that 
    $S_1 \subset T\setminus\partial T$ and $T \subset S_2\setminus\partial S_2$ is $T=S_2\setminus \partial S_2$, by Corollary \ref{Cor_CliqueSummary}. This subgraph $T$ lies in both $C(S_1)$ and $C(S_2)$. 
    
    Conversely, if the cliques intersect in a $T\cong\Delta_{m+3}$, this subgraph $T$ has 
    distance 1 to both the boundaries of $S_1$ and $S_2$; therefore, the boundaries 
    of $S_1$ and $S_2$ have distance 2.
\end{proof}

\noindent The preceding lemmata can be summarised in the following way:

\begin{cor}\label{Cor_Edges}
    For $S_1,S_2\in V(G_{n+1})$, the cliques $C(S_1)$ and $C(S_2)$ 
    intersect non-trivially if and only if $S_1$ and $S_2$ are 
    adjacent in $G_{n+1}$. Furthermore, for every 
    $n\in \Z_{\geq0}$, $G_n\cong k^nG$.
\end{cor}	

\noindent Now we can finally prove  our main theorem for \mammals.

\begin{theo}\label{Thm_ConvergentorDivergent}
    Let $G$ be a triangularly simply connected locally cyclic graph with 
    minimum degree at least $6$. 
    If there is an $m\geq 0$ such that 
    $\Delta_m$ cannot be embedded into $G$, the clique operator is 
    convergent on $G$.
\end{theo}
\begin{proof}
    If $\Delta_m$ cannot be embedded into $G$, this means $m\geq 2$ and  
    $G_{m-2}=G_m$ since the graphs $G_m$ and $G_{m-2}$ can only differ in vertices isomorphic to $\Delta_m$, which would be subgraphs of $G$.  
    But by Corollary \ref{Cor_Edges}, this means $k^{m}G\cong k^{m-2}G$, 
    which is the definition of the clique operator being convergent on $G$.
\end{proof}

\section{Coverings}\label{Sect_UniversalCover}
Up to this point, we only considered triangularly simply connected locally cyclic graphs.
Fortunately, any other locally cyclic graph is covered by a
simply connected one, to which we will apply Theorem \ref{Thm_ConvergentorDivergent}.

For the generalisation of the theory, we need results from
\cite{larrion2000locally} and \cite{rotman1973covering}, whose ways of notation
look incompatible at first glance. Instead of repeating and re-deriving
large parts of both, we show how to fit the definitions of 
\cite{larrion2000locally} into the setting of \cite{rotman1973covering}.

While one of the sources talks about \textit{simple graphs} with \textit{edges} and
\textit{triangles} (\cite[Section 1, p. 160]{larrion2000locally}), the
other one describes \textit{complexes} with \textit{1-simplices} and
\textit{2-simplices} (\cite[Section 1, p. 642]{rotman1973covering}). We
can transition from graphs to complexes by constructing the
\textit{triangular complex} $\K{G}$, which is defined in \cite{larrion2000locally}, as we did before.

\begin{lem}
    Let $G, G'$ be simple graphs. A vertex map $V(G) \to V(G')$
    defines a \textit{homomorphism} $G \to G'$ in the sense of 
    \cite[Section 2, p. 161]{larrion2000locally} if and only if
    it defines a \textit{map} $\K{G} \to \K{G'}$ in the sense
    of \cite[Section 1, p. 642]{rotman1973covering}.
\end{lem}
\begin{proof}
    Let $f: G \to G'$ be a homomorphism in the sense of
    \cite{larrion2000locally} and $\{u,v,w\}$ a triangle in $G$,
    i.\,e.\ a 2-simplex in $\K{G}$. By assumption, 
    $\{f(u),f(v)\}$, $\{f(u),f(w)\}$, and $\{f(v),f(w)\}$ are
    edges in $G'$. Thus, $\{f(u),f(v),f(w)\}$ is a triangle
    in $G'$, i.\,e.\ a 2-simplex in $\K{G'}$. The other implication is trivial.
\end{proof}

\noindent We will continue calling these maps graph homomorphisms.
We also take the next definition from \cite[Section 2, p. 162]{larrion2000locally}.

\begin{defi}\label{Def_CoveringMap}
    Let  $G, \tilde{G}$ be connected, simple graphs. A graph homomorphism  
    $p: \tilde{G} \to G$ is called a \emph{triangular covering map} 
    if it fulfils the triangle lifting property: For each triangle 
    $\{u,v,w\}\in G$ and each preimage $\tilde{u}$ of $u$, there exists a 
    (unique) triangle $\{\tilde{u},\tilde{v},\tilde{w}\}$ in 
    $\tilde{G}$ which is mapped to $\{u,v,w\}$ by $p$.
\end{defi}

\begin{lem}
    Let $G, \tilde{G}$ be connected, simple graphs and 
    $p: \tilde{G} \to G$ be a homomorphism. Then, $p$ is a triangular covering
    map  if
    and only if $(\K{\tilde{G}},p)$ is a 
    \textit{covering complex} (\cite[Section 2, p. 650]{rotman1973covering}).
\end{lem}
\begin{proof}
    We only need to show that the lifting properties are equivalent.

   \noindent For the first part, assume that $(\K{\tilde{G}},p)$ is
    a covering complex. By Theorem 2.1 (\cite[p. 651]{rotman1973covering}),
    $p$ has the unique path lifting property. It remains to show that
    $p$ has the triangle lifting property. Let $\{u,v,w\}$ be a triangle
    in $G$, i.\,e.\ a 2-simplex in $\K{G}$. Since
    $p^{-1}(\{u,v,w\})$ is the union of 2-simplices, every $\tilde{u} \in p^{-1}(u)$
    lies in some triangle of $\tilde{G}$.

    For the second part, assume that $p$ is a triangular covering map. Consider
    a 1-simplex $\{u,v\}$ in $\K{G}$ and let $\tilde{u} \in p^{-1}(u)$.
    By the unique edge lifting property of $p$ 
    (\cite[Section 2, p. 161]{larrion2000locally}), there is a unique 
    $\tilde{v} \in p^{-1}(v)$ adjacent to $\tilde{u}$. By the same argument,
    $\tilde{u}$ is unique with respect to $\tilde{v}$. Thus, $p^{-1}(\{u,v\})$
    splits into pairwise disjoint 1-simplices.

    Next, consider a 2-simplex $\{u,v,w\}$ in $\K{G}$. By the triangle
    lifting property (\cite[Section 2, p. 162]{larrion2000locally}), 
    $p^{-1}(\{u,v,w\})$ is the union of triangles in $\tilde{G}$. If two
    different triangles would intersect, the unique edge lifting property
    would be violated.
\end{proof}

\noindent We take the following definition from \cite[Section 3, p.\ 663]{rotman1973covering}.

\begin{defi}
	A \emph{universal covering complex} of $K$ is a covering complex $p\colon\tilde{K}\to K$ such that, for ever covering complex $q\colon \tilde{J}\to K$ there exists a unique map $h\colon \tilde{K}\to\tilde{J}$ making the following diagram commute:
	\begin{figure}[htbp]
		\centering
		\begin{tikzpicture} [scale=0.7]
		\HexagonalCoordinates{5}{5}
		\node (K') at (A01) {$\tilde{K}$};
		\node (J') at (A22) {$\tilde{J}$};
		\node (K) at (A30)  {$K$};
		\draw[dashed,->] (K')--(J')  node[draw=none,fill=none,font=\scriptsize,midway,above] {$h$}; 
		\draw[->] (K')--(K)  node[draw=none,fill=none,font=\scriptsize,midway,below] {$p$};
		\draw[->] (J')--(K)  node[draw=none,fill=none,font=\scriptsize,midway,right] {$q$};
		\end{tikzpicture}
	\end{figure}
\end{defi}
\noindent Universal covering complexes are unique up to isomorphism, 
and every connected complex has a universal covering 
complex (\cite[Section 3, p.\ 663]{rotman1973covering}).
We apply this to our graph setting.

\begin{defi}
	 Let $G, \tilde{G}$ be connected, simple graphs and 
	$p: \tilde{G} \to G$ be a triangular covering map. Then, $\tilde{G}$ is the \emph{universal cover} of $G$ if $(\K{\tilde{G}},p)$ is the universal covering complex of $\K{G}$.
\end{defi}

\noindent We would like to apply Proposition 3.2 from \cite{larrion2000locally} to the universal cover:
\begin{prop}\label{Prop_Galois_in_Cliques}
    Let $p: \tilde{G} \to G$ be Galois with group $\Gamma$. Then,
    $p_k: k\tilde{G} \to kG$ is also Galois with group $\Gamma_k \cong \Gamma$.
\end{prop}

\noindent To apply this proposition to the universal cover of an arbitrary locally cyclic graph of minimum degree at least $6$, we need to show that the universal cover always
defines a \textit{Galois triangular map} (\cite[Section 3, p. 165]{larrion2000locally}).\\
\begin{lem}\label{Lem_universal_covering_is_Galois}
    Let $G,\tilde{G}$ be connected (simple) graphs such that $\K{\tilde{G}}$
    is the universal cover of $\K{G}$. Then, the associated covering map
    $p\colon \K{\tilde{G}} \to \K{G}$ is Galois.
\end{lem}
\begin{proof}
    Define $\Gamma \isdef \{ \gamma \in \Aut(\tilde{G}) \mid p \circ \gamma = p \}$
    (the deck transformations from \cite[Section 3, p. 665]{rotman1973covering}).
    We need to show that $\Gamma$ acts transitively on each fibre of $p$. This is
    proven in Corollary 3.11 (\cite[p. 667]{rotman1973covering}) if $p$ is
    \textit{regular}. Since $p$ is a covering map from the universal cover,
    the regularity follows from the remark at the top of p. 666 in \cite{rotman1973covering}.
\end{proof}

\noindent Now, we can conclude that clique convergence of the universal cover transfers
to the original graph.
\begin{lem}\label{Lem_CoverConvergent}
    Let $G$ be a locally cyclic graph with $\delta(G)=6$, whose universal cover is \ccon.
     Then, $G$ is \ccon\enspace as well.
\end{lem}
\begin{proof}
	Let  $p\colon \tilde{G} \to G$  be the triangular covering map. By Lemma \ref{Lem_universal_covering_is_Galois}, $p$ is a Galois covering map with group $\Gamma$, the group of deck transformations. We define $p_{k^n}\colon k^n\tilde G\to k^nG$ by $p_{k^0}=p_k$ and $p_{k^n}(Q)=\{p_{k^{n-1}}(v)\mid v\in Q\}$ for every $n\geq 1$. By Proposition \ref{Prop_Galois_in_Cliques}, the maps $p_{k^n}$ are Galois with groups $\Gamma_{k^n}\cong \Gamma$.
	Since the universal cover is \ccon, there are $n,l\in\N$ such that $k^n\tilde G\cong k^{n+l}\tilde G.$  Thus, $p_{k^n}$ and $p_{k^{n+l}}$ are Galois covering maps with a group isomorphic to $\Gamma$. Therefore, 
        \begin{equation*}
            k^nG\cong k^n(\tilde{G})/\Gamma\cong k^{n+l}(\tilde{G})/\Gamma\cong k^{n+l}G,
        \end{equation*}
        and $G$ is \ccon\enspace as well.
\end{proof}	

\noindent Combining Lemma \ref{Lem_CoverConvergent} with Theorem \ref{Thm_ConvergentorDivergent}, we get a general criterium for convergence.
\begin{theo}
    If $G$ is a locally cyclic graph with $\delta(G)=6$ and there is an $m\geq 0$ such that 
    $\Delta_m$ cannot be embedded into the universal cover of $G$, then $G$ is \ccon.
\end{theo}

%Since we gave an explicit condition for clique divergence of the universal cover,
%we can also give a result for diverging universal covers.

\noindent If the graph $G$ is finite, we also have a criterium for divergence.

\begin{lem}
    Let $G$ be a finite locally cyclic graph with $\delta(G)=6$ whose universal cover is
    \cdiv. Then, $G$ is $6$-regular.
\end{lem}
\begin{proof}
    Let $\tilde{G}$ be the universal cover and $p: \tilde{G} \to G$
    the universal covering map.
    Since the universal cover diverges, there is a 
    $\Delta_m \subset \tilde{G}$
    for any $m \geq 1$ which is mapped to $G$ via $p$. 
    Since $G$ is finite, the maximal length of a facet-path
    between any two vertices can be bounded by a finite
    number $d$. 
    
    Now, consider $\Delta_{3d+3} \subset \tilde{G}$. For any vertex
    $x \in G$, there is a facet-path between $p(d+1,d+1,d+1)$
    and $x$ with length at most $d$. Since $p$ is a covering map
    (compare Definition \ref{Def_CoveringMap}), this facet-path
    lifts to a facet-path in $\tilde{G}$. All vertices
    in $\Delta_{3d+3}$ with distance at most $d$ from $(d+1,d+1,d+1)$
    are inner vertices of $\Delta_{3d+3}$; thus, the vertex $x$ has the same 
    degree as such an inner vertex, namely, 6.
\end{proof}

\noindent Since by \cite[Theorem 1.1]{larrion2000locally},
a locally cyclic graph which is $6$-regular is \cdiv, we state our main theorem.

\begin{cor}[Main result]
    Let $G$ be a locally cyclic graph with minimum degree at least $6$.
    \begin{enumerate}
        \item For a finite graph $G$, the clique graph operator diverges on $G$ if and only if $G$ has only vertices of degree $6$.
        \item For an infinite graph $G$, if there exists an $m\geq 0$, such that 
            $\Delta_m$ cannot be embedded into the universal cover of $G$, the clique operator is convergent on $G$.
    \end{enumerate}
\end{cor}

Whether an infinite locally cyclic $G$ can be convergent, even though
its universal cover diverges, is still an open question.

\section{Further Research}\label{Sect_FurtherResearch}

\noindent In our research, we were able to decide which finite locally 
cyclic graphs with minimum degree $\delta=6$ are \ccon\enspace and 
which are \cdiv. But we are not able to decide this for infinite graphs, 
not even if they are triangularly simply connected. To prove in an 
analogous way that every \mammal\footnote{i.\,e. a triangularly simply 
	connected locally cyclic graphs with minimum degree $\delta=6$} %\enspace
  which 
contains a subgraph isomorphic to $\Delta_m$, for every $m$, is \cdiv, 
it would be necessary to show that $k^{n}G\subsetneq k^{n+l}G$ implies 
$k^{n}G\not\cong k^{n+l}G$. Even if this was proven, our classification of 
$k$-convergence would not be finished, since an 
infinite graph with a \cdiv\enspace universal cover can itself be \ccon.

Our work shows that explicit consideration of the clique
dynamics can be fruitful. It would be interesting to know whether
this approach gives feasible results for smaller minimum degrees.

% flatex input end: [CliqueGraphConstruction.tex]

%\input{CliqueGraphPart1.tex}
%\input{CliqueGraphPart2.tex}
%\input{AdditionalStuff.tex}

% flatex input: [Appendix.tex]
\newpage
\appendix

\section{Definitions}
\begin{defi}
	For a graph $G=(V_G,E_G)$, 
	the \emph{closed neighbourhood} of $M\subset V_G$ in $G$ is given by the induced subgraph
	\begin{equation*}
	\mathemph{\neig{G}{M}}\isdef
	G[y\in V_G\mid y\in M\text{ or }\exists x\in M\colon xy\in E_G],
	\end{equation*}
	and the \emph{common neighbourhood} of $M$ is 
	\begin{equation*}
	\mathemph{\cneig{G}{M}}\isdef G[y\in V_G\mid y\in M\text{ or } \forall x\in M\colon xy\in E_G ].
	\end{equation*}
	For a subgraph $H$ of $G$, $\mathemph{\neig{G}{H}}\isdef \neig{G}{V_H}$ and 
	$\mathemph{\cneig{G}{H}}\isdef \cneig{G}{V_H}$. Furthermore, for a vertex $v\in V_G$,
	the \emph{closed neighbourhood} is given by $\mathemph{\neig{G}{v}}
        \isdef \neig{G}{\{v\}}$ and the \emph{open neighbourhood}
	is given by $N_G(v)=G[y\in V_G\mid vy\in E_G].$
	
	%If $H$ is a subgraph of $G$, the neighbourhood of $H$ in $G$ is given by $N_G[H]=G[y\in V_G\mid \exists x\in V_H\colon xy\in E_G]$.
	%The common neighbourhood of $H$ in $G$ is given by $ N_G^{\cap}[H]=G[y\in V_G\mid \forall x\in V_H\colon xy\in E_G]$.
	For the two graphs $G = (V_G, E_G)$ and $H = (V_H, E_H)$,
	the graph $\mathemph{G\setminus H}$ is defined as 
	$(V_G\ohne V_H, \{ xy \in E_G \ohne E_H \mid x \not\in V_H \wedge y \not\in V_H \})$.

	%If $H$ is a subgraph of $G$, the neighbourhood of $H$ in $G$ is given by 
	%$N_G(H)=G[xy\in E_G\mid x\in V_H]$. 
	
\end{defi}	
\begin{defi}
    \noindent A \emph{graph homomorphism} $\phi\colon G\to H$ is any 
    adjacency-preserving vertex map, i.\,e. 
    $uv\in E(G)\ \Rightarrow\ f(u)f(v)\in E(H)$, see 
    \cite[Section 2, p. 161]{larrion2000locally}. Injective homomorphisms 
    are called \emph{monomorphisms}.
    An \emph{isomorphism} is a bijective homomorphism whose inverse
    is also a homomorphism.
\end{defi}

\section{Proofs from Topology}\label{Sec_Proofsfromtopology}

\begin{proof}[Proof of Lemma \ref{lem_topology}]
	In all cases, we start with a path $v_0v_1\dots v_k$ with $v_0,v_k\in \partial S$ and
	$v_1,\dots,v_{k+1} \not\in S$ such that none of the edges $v_iv_{i+1}$, for $0 \leq i < k$,
	lies in $S$. In the first case, we aim for a contradiction, in the second and third we show the claims directly.
	
	Since $G$ is simply connected, both the path and $S$ lie in a common
	planar subgraph $U$ of $G$, such that every bounded face is a triangle.	
	%    Since $G$ is simply connected, both the path and $S$ lie in a common
	%    planar region.\todo{What is a planar region? I cannot find a short desciption.} 
	Henceforth, we consider all paths as paths in the plane.

	There are two paths along $\partial S$ that connect $v_k$ to $v_0$.
	By \cite[Corollary 1.2]{Thomassen_KuratowskiTheorem}, one of those together with $v_0v_1\dots v_k$ bound a disc containing the other one. These two paths cannot be
	the paths along $\partial S$, since $v_0\dots v_k$ does not lie in $S$.
	The path which lies ``inside'' will be denoted by $v_k=s_0s_1\dots s_m=v_0$,
	the other one by $v_kv_{k+1}\dots v_r=v_0$.
	
	We define $\alpha_i$ as the inner 
	\emph{facet-degree}\footnote{Which counts the
		number of facets and not the number of edges} at $v_i$ of the path
	$v_0\dots v_ks_1\dots s_m$ for every $0 \leq i \leq k$ (thus, $\alpha_0$
	and $\alpha_k$ are the facet-degrees between the path and $\partial S$).
	To prove the lemma, we focus on the path $v_0\dots v_kv_{k+1}\dots v_r$
	in more detail. We denote the inner facet-degree at $v_j$ by $\beta_j$.
	This situation is displayed in Figure \ref{fig_top}.
	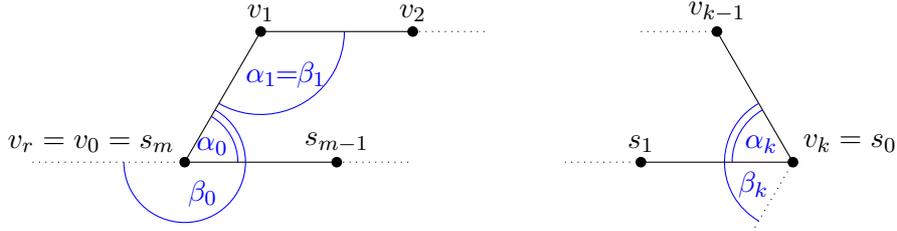
\begin{figure}[htbp]
		\begin{center}
			\begin{tikzpicture}
			\HexagonalCoordinates{6}{2}
			
			\draw (A21) -- (A11) -- (A12) -- (A22);
			\draw (A41) -- (A51) -- (A42);
                            \draw[dotted] (A21) -- ($(A31)!0.5!(A21)$);
			\draw[dotted] (A01) -- (A11);
                            \draw[dotted] ($(A31)!0.5!(A41)$) -- (A41);
                            \draw[dotted] (A22) -- ($(A22)!0.5!(A32)$);
                            \draw[dotted] ($(A32)!0.5!(A42)$) -- (A42);
			\draw[dotted] (barycentric cs:A50=1,A51=1) -- (A51);
			
			\foreach \v/\r/\n in {11/above left/$v_r = v_0 = s_m$,
				12/above/$v_1$, 22/above/$v_2$, 42/above/$v_{k-1}$, 21/above/$s_{m-1}$,
				41/above/$s_1$, 51/above right/$v_k=s_0$}{
				\fill[black] (A\v) circle (2pt);
				\node[\r] at (A\v) {\n};
			}
			
			% position, radius, start angle, end angle, label
			\foreach \p/\r/\s/\e/\n in { 
				A11/0.8/-180/60/$\beta_0$, A11/0.7/0/60/$\alpha_0$,
				A12/1.1/-120/0/$\alpha_1\!\!=\!\!\beta_1$,
				A51/0.8/120/180/$\alpha_k$, A51/0.9/120/240/}{
				\draw[blue] ($(\p)+(\s:\r)$) arc (\s:\e:\r);
				\node[blue] at ($(\p)+(0.5*\s+0.5*\e:0.6*\r)$) {\n};
			}
			\node[blue] at ($(A51)+(210:0.6)$) {$\beta_k$};
			
			\end{tikzpicture}
		\end{center}
		\caption{Illustration of a path starting at a boundary vertex of $\Delta_m\cong S\subseteq G$ and ending at a corner vertex}
		\label{fig_top}
	\end{figure}
	
	Since this path bounds a disc in the plane, Lemma 4.1.5 from 
	\cite{Baumeister_PhD} is applicable and gives
	\begin{equation*}
	6 = \sum_{v \text{ inner vertex}}(6 - \deg(v)) + \sum_{j=0}^{r-1}(3-\beta_j).
	\end{equation*}
	Since $\deg(v) =6$ for all vertices in $G$, we obtain
	\begin{equation*}
	6 \leq \sum_{j=0}^{r-1}(3-\beta_j) = (3-\beta_0) + \sum_{i=1}^{k-1}(3-\alpha_i) + (3-\beta_k) + \sum_{j=k+1}^{r-1}(3-\beta_j).
	\end{equation*}
	If a vertex $v_j$ with $k < j < r$ lies in a corner of $S$,
	we have $\beta_j = 1$; otherwise, $\beta_j=3$. Thus, it remains to
	analyse $\beta_0$ and $\beta_k$.
	
	If $v_0$ is  a corner vertex, we have $\beta_0 = 1 + \alpha_0$;
	otherwise, we have $\beta_0 = 3 + \alpha_0$. 
	Whichever case applies, we can rewrite the inequality into
	\begin{equation*}
	6 \leq \sum_{i=0}^k (3 - \alpha_i) - 6 + 2c,
	\end{equation*}
	where $c \in \{0,1,2,3\}$ is the number of corner vertices in 
	$\{v_k,v_{k+1},\dots,v_r\}$ (note the inclusion of $v_k$ and
	$v_r$ in this set).
	With this inequality,
	we proceed through the three cases of the lemma. Recall that
	$\alpha_0 \geq 1$ and $\alpha_k \geq 1$ have to hold (otherwise
	the edge $v_0v_1$ or the edge $v_{k-1}v_k$ lies in $S$).
	\begin{enumerate}
	    \item For the path $v_0v_1$ we obtain $6 \leq 2c - \alpha_0 - \alpha_1$,
		which has no solutions. Thus, there can be no edge between these vertices
		that does not already lie in $S$.
	    \item For the path $v_0v_1v_2$, we obtain $3 \leq 2c - \alpha_0 - \alpha_1 - \alpha_2$.
		If $\alpha_1 = 0$, we obtain $v_0 = v_2$. Otherwise, the only possible solution
		is $c=3$ and $\alpha_0=\alpha_1=\alpha_2=1$. This already implies that
		$\{v_0,v_1,v_2\}$ is a facet of $G$.
	    \item For the path $v_0v_1v_2v_3$, we obtain $2c \geq \alpha_0+\alpha_1+\alpha_2+\alpha_3$.
		Since the path is non-repeating, we have $\alpha_1 \geq 1$
		and $\alpha_2 \geq 1$; thus $c \in \{2,3\}$.
		
		Now, $\alpha_1 = 1$ implies the facet $\{v_0,v_1,v_2\}$. In particular,
		$v_0v_2v_3$ is a path. With part \eqref{top_b} of this lemma,
		we conclude that $\{v_0,v_2,v_3\}$ is a facet, in contradiction
		to our assumption. The same argument applies if $\alpha_2=1$.
		Thus, both of them have to be at least 2.
		
		Then, the only solution is $c= 3$ with $\alpha_0 = \alpha_3 = 1$
		and $\alpha_1 = \alpha_2 = 2$. Since $\alpha_0 = 1$, the triple $\{v_0,s_{m-1},v_1\}$ forms a facet.
		Since $\alpha_1 = 2$, the triple
		$\{v_1,s_{m-1},v_2\}$ also has to be a facet. For $v=s_{m-1}$, this was the claim
		that needed to be shown. \qedhere
	\end{enumerate}
\end{proof}

\begin{proof}[Proof of Lemma \ref{Lem_NeighbourhoodIntersection}]
	Assume to the contrary that there is an 
	$x \in (\neig{G}{S_1} \cap \neig{G}{S_2} \cap \neig{G}{S_3} ) \ohne S$.
	Since $S_i \subset S$, we conclude $x \in \neig{G}{S} \ohne S$. Without
	loss of generality, $x$ is adjacent to $(t,m-t,0)$ for some $0 \leq t < m$.
	We permute the coordinates, such that $t$ is maximal among all edges.
	\begin{enumerate}
		\item Case: $t > 0$:\\
		Since $(t,m-t,0) \not\in \Delta_{m-1}^{001}$, the vertex $x$ has
		to be adjacent to a boundary vertex of $\Delta_{m-1}^{001}$ as well,
		say $(s,0,m-s)$ for some $0 \leq s < m$ (see Figure \ref{t_larger_zero}).
		
		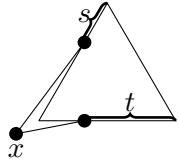
\begin{figure}[htbp]
                    \centering
                    \begin{tikzpicture}[scale=0.3,decoration={brace}]
                        \HexagonalCoordinates{4}{4}
                        
                        \draw
                        (A11) -- (A41) -- (A14) -- (A11)
                        (D00) -- (A13)
                        (D00) -- (A21);
                        \draw [decorate,line width=1pt] (A13) -- (A14);
                        \draw [decorate,line width=1pt] (A21) -- (A41);
                        \node at (D03) {$s$};
                        \node[above] at (A31) {$t$};
                        
                        \node[shape=circle, fill=black,scale=0.5] at (A13) {};
                        \node[shape=circle, fill=black,scale=0.5] at (A21) {};
                        \node[shape=circle, fill=black,scale=0.5] at (D00) {};
                        \node[below] at (D00) {$x$};
                    \end{tikzpicture}
                    \caption{Illustration of common neighbour $x$.}\label{t_larger_zero}
		\end{figure}
		
		By Lemma \ref{lem_topology}\eqref{top_b}, the vertices
		$(t,m-t,0)$ and $(s,0,m-s)$ have to be adjacent. This is
		only possible for $t = s = m-1$. But both facets incident
		to the edge $\{(m-1,1,0), (m-1,0,1)\}$ already lie in $S$ (for $m > 1$),
		contradicting $x \not\in S$.
		\item Case: $t=0$:\\
		Since $x$ is only adjacent to corner vertices (otherwise
		we would be in the case $t>0$), it has to be adjacent to
		$(0,m,0)$ and $(0,0,m)$ as well (to lie in each $\neig{G}{S_i}$).
		By Lemma \ref{lem_topology}\eqref{top_b}, this implies the
		adjacency of the corner vertices, i.\,e. $m=1$. But then, the 
		neighbourhood of $x$ contains a circle of length $3$, in 
		contradiction to neighbourhoods being circles of length at least $6$.\qedhere
		
	\end{enumerate}
\end{proof}

\begin{proof}[Proof of Lemma \ref{Lem_NeighbourhoodCycle}]
	If $m=0$, this is the definition of locally cyclic. For $m\geq 1$,
	we enumerate the boundary vertices of $S$ in cyclic order by
	$b_1b_2\dots b_{3m}$. For each adjacent pair $(b_i,b_{i+1})$,
	there is exactly one face containing $\{b_i,b_{i+1}\}$ and not
	lying in $S$. Call the final corner of this face $n_{i,i+1}$. 
	If $n_{i,i+1} \in S$, Lemma \ref{lem_topology} (\ref{top_a})
	would imply that the full face lied in $S$.
	
	With this notation, $N_G(b_i)\ohne S$ is the path
	$n_{i-1,i}x_1x_2\dots x_kn_{i,i+1}$ (there are no further edges
	between these vertices since the neighbourhood of $b_i$ is a cycle).
	None of the $x_i$ lies in $S$, since Lemma \ref{lem_topology} 
	(\ref{top_a}) would imply further boundary edges of $S$, in
	contradiction to our assumption. This situation is illustrated
        in Figure \ref{Fig_localNeighbourhood}.

        \begin{figure}[htbp]
            \centering
            \begin{tikzpicture}[scale=1]
                \HexagonalCoordinates{3}{2}
                \def\rad{1.8}
                \def\ang{40}
                \coordinate (N1) at ($(A11)+(120-\ang:\rad)$);
                \coordinate (Nk) at ($(A11)+(\ang:\rad)$);
                
                \pgfsetfillpattern{north west lines}{blue}
                \fill (A00) -- (A20) -- (A11) -- (A01) -- (A00);
                
                \foreach \p/\r/\l in {A01/above left/$b_{i-1}$, A11/below left/$b_i$, A20/right/$b_{i+1}$,
                    A02/above/{$n_{i-1,i}$}, N1/above/$x_1$, Nk/right/$x_k$, A21/right/{$n_{i,i+1}$}}{
                        \node[\r,fill=white] at (\p) {\l};
			\fill[black] (\p) circle (2pt);
                    }
                \node[blue,fill=white] at (U00) {$S$};

                \draw
                    (A00) -- (A01) -- (A11) -- (A20)
                    (A01) -- (A02) -- (A11) -- (A21) -- (A20)
                    (A02) -- (N1) -- (A11) -- (Nk) -- (A21);
                \draw[dotted] (Nk) arc (\ang:120-\ang:\rad);
            \end{tikzpicture}
            \caption{Part of a local neighbourhood}\label{Fig_localNeighbourhood}
        \end{figure}
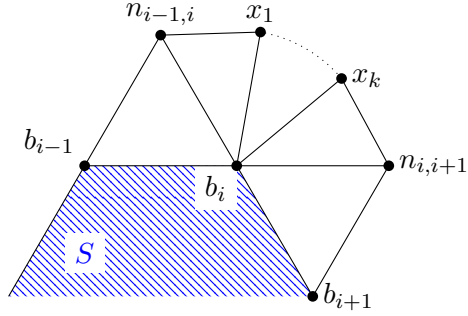
	
	By Lemma \ref{lem_topology}\eqref{top_c}, any edge between
	vertices in $\neig{G}{S}\ohne S$ already lies in one $N_G(b_i)\ohne S$.
	Combining these paths gives the desired cycle.
\end{proof}

\section{Proofs from Properties of Triangles}\label{Sect_propoftriang}
\begin{proof}[Proof of Lemma \ref{Lem_TriangleInclusionOne}]
	Let $\Delta_{m-1}^{\vec{t}} \to \Hex_m$ be a triangle inclusion map
	whose image lies in $\Delta_m$. Thus, each component
	of $\vec{t}$ has to be non-negative, leaving 
	$\vec{t} \in V_1\cap\Z_{\geq 0}^3=\vec{E}$.
	
	It remains to consider those $\Delta_{m-1} \cong S \subset \Delta_m$
	that are not the image of a triangle inclusion map. The boundary
	of such an $S$ consists of three straight paths of length $m-1$.
	By Remark \ref{Rem_StraightPathsInTriangle}, there
	are six such paths along the boundary of $\Delta_m$ and three such paths in
	the interior (each given by all the vertices for which one fixed has value 1). 
	The boundary paths can only lie in one $\Delta_{m-1} \cong S \subset \Delta_m$.
	Since the triangle inclusion maps ``use'' two boundary paths and one
	interior path each, the only remaining possibility is combining
	the three interior paths into a $\Delta_{m-1}$. But this is only possible
	if the paths meet at the boundary (where one component is 0). Thus,
	the component sum $m$ has to be 2.
\end{proof}

\begin{proof}[Proof of Lemma \ref{Lem_TriangleInclusionTwo}]
	Let $\Delta_{m-2}^{\vec{f}}\colon \Delta_{m-2} \to \Hex_m$ be a 
	triangle inclusion map
	whose image lies in $\Delta_m$. Thus, each component of $\vec{t}$ has
	to be non-negative and thus  $\vec{f}\in V_2\cap\Z_{\geq 0}^3=\vec{E}+\vec{E}$.
	For $m=2$, all $\Delta_0\cong S\subseteq \Delta_m$ are possible images.
	
	It remains to consider those $\Delta_{m-2} \cong S \subset \Delta_m$
	that are not the image of a triangle inclusion map. In this case 
	$m\geq 3$. The boundary of
	such an $S$ consists of three straight paths of length $m-2$.
	Remark \ref{Rem_StraightPathsInTriangle} describes all paths of this kind inside $\Delta_{m}$.
	Up to the group action (Subsection \ref{Subsect_Hexagonal}),
	we only need to look at four of these paths:
	\begin{enumerate}
		\item The path $\alpha_{|\{0,\dots,m-2\}}$
		can only be the boundary of one $\Delta_{m-2}$ and it already
		is the boundary of $\Delta_{m-2}^{(2,0,0)}$.
		\item The path $\alpha_{|\{1,\dots,m-1\}}$
		can only be the boundary of one $\Delta_{m-2}$ and it already
		is the boundary of $\Delta_{m-2}^{(1,1,0)}$.
		\item\label{case} The path $\beta_{|\{0,\dots,m-2\}}$ is
		the boundary of $\Delta_{m-2}^{(1,0,1)}$ (whose third component
		is higher). There can only be a $\Delta_{m-2}$ with lower third
		component if $m-2 \leq 1$, implying $m=3$. The triangle has corner
		vertices $(2,0,1)$, $(1,1,1)$, and $(2,1,0)$.
		\item The path $\gamma$ is the 
		boundary of $\Delta_{m-2}^{(0,0,2)}$ (whose third component is
		higher). There can only be a $\Delta_{m-2}$ with lower third
		component if $m-2 \leq 2$. The case $m=3$ gives a triangle which 
		is brought to the triangle of (\ref{case}.) by the group action.
		The case $m=4$ gives vertices $(2,0,2)$, $(0,2,2)$, and $(2,2,0)$.
	\end{enumerate}
	Applying the group action to these $\Delta_{m-2}$ gives the desired results.
\end{proof}

% flatex input end: [Appendix.tex]

%\input{AdditionalStuff.tex}

%FLATEX-REM:\bibliography{referenzen}

\begin{thebibliography}{1}

\bibitem{Baumeister_PhD}
M.~Baumeister.
\newblock {\em {R}egularity aspects for combinatorial simplicial surfaces}.
\newblock Dissertation, RWTH Aachen University, Aachen, 2020.
\newblock Veröffentlicht auf dem Publikationsserver der RWTH Aachen
  University; Dissertation, RWTH Aachen University, 2020.

\bibitem{hedetniemi1972line}
S.~Hedetniemi and P.~Slater.
\newblock Line graphs of triangleless graphs and iterated clique graphs.
\newblock In {\em Graph theory and Applications}, pages 139--147. Springer,
  1972.

\bibitem{larrion1999clique}
F.~Larri{\'o}n and V.~Neumann-Lara.
\newblock Clique divergent graphs with unbounded sequence of diameters.
\newblock {\em Discrete Mathematics}, 197:491--501, 1999.

\bibitem{larrion2000locally}
F.~Larri{\'o}n and V.~Neumann-Lara.
\newblock Locally c6 graphs are clique divergent.
\newblock {\em Discrete Mathematics}, 215(1-3):159--170, 2000.

\bibitem{larrion2003clique}
F.~Larri{\'o}n, V.~Neumann-Lara, and M.~Piza{\~n}a.
\newblock Clique convergent surface triangulations.
\newblock {\em Mat. Contemp}, 25:135--143, 2003.

\bibitem{larrion2002whitney}
F.~Larri{\'o}n, V.~Neumann-Lara, and M.~A. Piza{\~n}a.
\newblock Whitney triangulations, local girth and iterated clique graphs.
\newblock {\em Discrete Mathematics}, 258(1-3):123--135, 2002.

\bibitem{larrion2006graph}
F.~Larri{\'o}n, V.~Neumann-Lara, and M.~A. Piza{\~n}a.
\newblock Graph relations, clique divergence and surface triangulations.
\newblock {\em Journal of Graph Theory}, 51(2):110--122, 2006.

\bibitem{rotman1973covering}
J.~Rotman.
\newblock Covering complexes with applications to algebra.
\newblock {\em The Rocky Mountain Journal of Mathematics}, 3(4):641--674, 1973.

\bibitem{Thomassen_KuratowskiTheorem}
C.~Thomassen.
\newblock Kuratowski's theorem.
\newblock {\em Journal of Graph Theory}, 5(3):225--241, 1981.

\end{thebibliography}
%*flatex input: [CliqueGraphMain.bbl]

% flatex input end: [CliqueGraphMain.bbl]
%FLATEX-REM:\bibliographystyle{abbrv}

%\input{WhitneyReferences.tex}
\end{document}